\input{aipcheck}

\documentclass[
    ,final            % use final for the camera ready runs
  ,numberedheadings % uncomment this option for numbered sections
  ]
  {aipproc}

\layoutstyle{8x11double}

\usepackage{algorithmic}
\usepackage{algorithm}

\usepackage{amsmath,amsfonts,amssymb,mathrsfs,amsthm}

\usepackage[pdf]{pstricks}
\usepackage{auto-pst-pdf}

\def\R		{\mathbb{R}}

\def\d		{\:\mathrm{d}}

\renewcommand{\vec}[1]{\mathbf{#1}}

\newcommand{\clifford}[1]{\ensuremath{C\kern-0.1em\ell_{#1}}}% #1 stands for the values p,q.
\def\e		{\vec{e}}

\def\a		{\vec{a}}
\def\b		{\vec{b}}
\def\c		{\vec{c}}
\def\A		{\vec{A}}
\def\B		{\vec{B}}
\def\x		{\vec{x}}
\def\y		{\vec{y}}

\def\u		{\vec{u}}
\def\v		{\vec{v}}

 \newtheorem{thm}{Theorem}[section]
 \newtheorem{cor}[thm]{Corollary}
 \newtheorem{lem}[thm]{Lemma}
 
 \theoremstyle{definition}
 \newtheorem{defn}[thm]{Definition}
 \theoremstyle{remark}
 \newtheorem{rem}[thm]{Remark}
 \newtheorem*{ex}{Example}
 \numberwithin{equation}{section}
 
\usepackage{mathptmx}       % selects Times Roman as basic font
\usepackage{helvet}         % selects Helvetica as sans-serif font
\usepackage{courier}        % selects Courier as typewriter font
\usepackage{type1cm}        % activate if the above 3 fonts are
\usepackage{makeidx}         % allows index generation
\usepackage{graphicx}        % standard LaTeX graphics tool
\usepackage{multicol}        % used for the two-column index
\usepackage[bottom]{footmisc}% places footnotes at page bottom

\makeindex 
\begin{document}

\title[Total Rotations and Geometric Correlation]{Detection of Total Rotations on 2D-Vector Fields with Geometric Correlation}

%<Replace this text with PACS numbers; choose from this list:
                %\texttt{http://www.aip..org/pacs/index.html}>}

\author{Roxana Bujack}
{address={%
Universit\"at Leipzig,
Institut f\"ur Informatik,
Johannisgasse 26,
04103 Leipzig, Germany
}}
\author{Gerik Scheuermann}
{address={%
Universit\"at Leipzig,
Institut f\"ur Informatik,
Johannisgasse 26,
04103 Leipzig, Germany
}}
\author{Eckhard Hitzer}
{address={%
University of Fukui,
Department of Applied Physics,
3-9-1 Bunkyo,
Fukui 910-8507, Japan
}}

\begin{abstract}
Correlation is a common technique for the detection of shifts. Its generalization to the multidimensional geometric correlation in Clifford algebras additionally contains information with respect to rotational misalignment. It has been proven a useful tool for the registration of vector fields that differ by an outer rotation. 
\par
In this paper we proof that applying the geometric correlation iteratively has the potential to detect the total rotational misalignment for linear two-dimensional vector fields. We further analyze its effect on general analytic vector fields and show how the rotation can be calculated from their power series expansions.
\end{abstract}

\maketitle

%%%%%%%%%%%%%%%%%%%%%%%%%%%%%%%%%%%%%%%%%%%%
%% MAINMATTER
%%%%%%%%%%%%%%%%%%%%%%%%%%%%%%%%%%%%%%%%%%%%

%---------------------------------------------------------------------------------------------------------------------------------
\section{Introduction}
%---------------------------------------------------------------------------------------------------------------------------------
\cite{LR10} In signal processing correlation is one of the elementary techniques to measure the similarity of two input signals. 
It can be imagined like sliding one signal across the other and multiplying both at every shifted location. The point of registration is the very position, where the normalized cross correlation function takes its maximum, because intuitively explained there the integral is built over squared and therefore purely positive values. For a detailed proof compare \cite{RK82}.
Correlation is widely used for signal analysis, image registration, pattern recognition, and feature extraction \cite{BRO92,ZIT03}. 
\par 
For quite some time the generalization of this method to multivariate data has only been parallel processing of the single channel technique. Multivectors, the elements of geometric or Clifford algebras $\clifford{p,q}$ \cite{C1878,HS84} have a natural geometric interpretation. So the analysis of multidimensional signals expressed as multivector valued functions is a very reasonable approach.
\par
Scheuermann made use of Clifford algebras for vector field analysis in \cite{Sch99}. Together with Ebling \cite{Ebl03,Ebl06} they applied geometric convolution and correlation to develop a pattern matching algorithm. They were able to accelerate it by means of a Clifford Fourier transform and the respective convolution theorem.
\par
At about the same time Moxey, Ell, and Sangwine \cite{MSE02,MSE03} used the geometric properties of quaternions to represent color images,  interpreted as vector fields. They introduced a generalized hypercomplex correlation for quaternion valued functions.
Moxey et. al. state in \cite{MSE03}, that the hypercomplex correlation of translated and outer rotated images will have its maximum peak at the position of the shift and that the correlation at this point also contains information about the outer rotation. From this they were able to approximately correct rotational distortions in color space. 
\par
In \cite{BSH12a} we extended their work and ideas analyzing vector fields with values in the Clifford algebra $\clifford{3,0}$ and their copies produced from outer rotations. We proved that iterative application of the rotation encoded in the cross correlation at the point of registration completely eliminates the outer misalignment of the vector fields.
\par
In this paper we go one step further and analyze if iteration can not only lead to the detection of outer rotations but also to the detection of total rotations of vector fields.
\par
The term rotational misalignment with respect to multivector fields is ambiguous. We distinguish three cases, visualized for a simple example in Figure \ref{f:1}. Let $\operatorname{R} _{\alpha}$ be an operator, that describes a mathematically positive rotation by the angle $\alpha$. 
\par
\begin{figure}[ht]
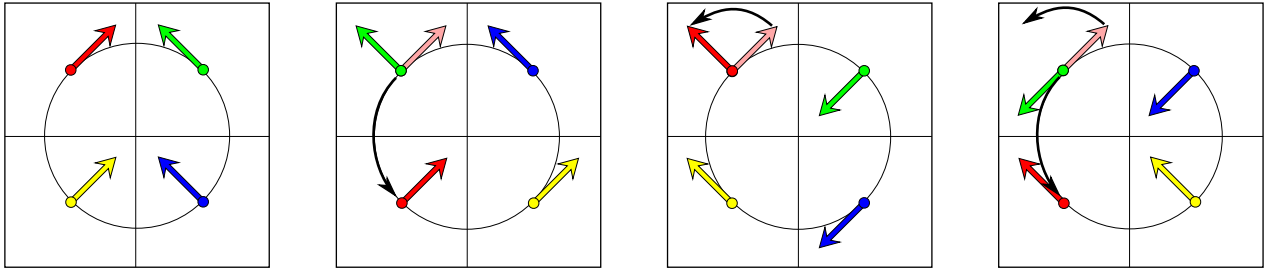
\label{f:1}
\begin{minipage}{0.25\textwidth}
\centering
\psset{unit=1pt}
  %LaTeX with PSTricks extensions
%%Creator: inkscape 0.47
%%Please note this file requires PSTricks extensions
%\psset{xunit=.5pt,yunit=.5pt,runit=.5pt}
\begin{pspicture}(102,102)
{
\newrgbcolor{curcolor}{1 1 1}
\pscustom[linestyle=none,fillstyle=solid,fillcolor=curcolor]
{
\newpath
\moveto(1,101)
\lineto(101,101)
\lineto(101,1)
\lineto(1,1)
\lineto(1,101)
\closepath
}
}
{
\newrgbcolor{curcolor}{0 0 0}
\pscustom[linewidth=0.5,linecolor=curcolor]
{
\newpath
\moveto(1,101)
\lineto(101,101)
\lineto(101,1)
\lineto(1,1)
\lineto(1,101)
\closepath
}
}
{
\newrgbcolor{curcolor}{1 1 1}
\pscustom[linestyle=none,fillstyle=solid,fillcolor=curcolor]
{
\newpath
\moveto(86.0166576,50.76017173)
\curveto(86.0166576,31.4302046)(70.3466235,15.76017013)(51.01665683,15.76017013)
\curveto(31.68669015,15.76017013)(16.01665605,31.4302046)(16.01665605,50.76017173)
\curveto(16.01665605,70.09013886)(31.68669015,85.76017333)(51.01665683,85.76017333)
\curveto(70.3466235,85.76017333)(86.0166576,70.09013886)(86.0166576,50.76017173)
\closepath
}
}
{
\newrgbcolor{curcolor}{0 0 0}
\pscustom[linewidth=0.09999999,linecolor=curcolor]
{
\newpath
\moveto(86.0166576,50.76017173)
\curveto(86.0166576,31.4302046)(70.3466235,15.76017013)(51.01665683,15.76017013)
\curveto(31.68669015,15.76017013)(16.01665605,31.4302046)(16.01665605,50.76017173)
\curveto(16.01665605,70.09013886)(31.68669015,85.76017333)(51.01665683,85.76017333)
\curveto(70.3466235,85.76017333)(86.0166576,70.09013886)(86.0166576,50.76017173)
\closepath
}
}
{
\newrgbcolor{curcolor}{1 0 0}
\pscustom[linestyle=none,fillstyle=solid,fillcolor=curcolor]
{
\newpath
\moveto(24.50853795,75.67407342)
\lineto(38.65067355,89.81620902)
\lineto(35.11513965,90.5233158)
\lineto(42.89331423,92.64463614)
\lineto(40.77199389,84.86646156)
\lineto(40.06488711,88.40199546)
\lineto(25.92275151,74.25985986)
\lineto(24.50853795,75.67407342)
\closepath
}
}
{
\newrgbcolor{curcolor}{0 0 0}
\pscustom[linewidth=0.1,linecolor=curcolor]
{
\newpath
\moveto(24.50853795,75.67407342)
\lineto(38.65067355,89.81620902)
\lineto(35.11513965,90.5233158)
\lineto(42.89331423,92.64463614)
\lineto(40.77199389,84.86646156)
\lineto(40.06488711,88.40199546)
\lineto(25.92275151,74.25985986)
\lineto(24.50853795,75.67407342)
\closepath
}
}
{
\newrgbcolor{curcolor}{1 0 0}
\pscustom[linestyle=none,fillstyle=solid,fillcolor=curcolor]
{
\newpath
\moveto(27.33696695,74.25986894)
\curveto(26.55591481,73.4788168)(25.28958199,73.47881391)(24.50853343,74.25986247)
\curveto(23.72748486,75.04091103)(23.72748776,76.30724385)(24.5085399,77.08829599)
\curveto(25.28959203,77.86934813)(26.55592485,77.86935102)(27.33697342,77.08830246)
\curveto(28.11802198,76.3072539)(28.11801908,75.04092108)(27.33696695,74.25986894)
\closepath
}
}
{
\newrgbcolor{curcolor}{0 0 0}
\pscustom[linewidth=0.1,linecolor=curcolor]
{
\newpath
\moveto(27.33696695,74.25986894)
\curveto(26.55591481,73.4788168)(25.28958199,73.47881391)(24.50853343,74.25986247)
\curveto(23.72748486,75.04091103)(23.72748776,76.30724385)(24.5085399,77.08829599)
\curveto(25.28959203,77.86934813)(26.55592485,77.86935102)(27.33697342,77.08830246)
\curveto(28.11802198,76.3072539)(28.11801908,75.04092108)(27.33696695,74.25986894)
\closepath
}
}
{
\newrgbcolor{curcolor}{0 0 0}
\pscustom[linewidth=0.1,linecolor=curcolor]
{
\newpath
\moveto(1,50.5)
\lineto(101,50.5)
}
}
{
\newrgbcolor{curcolor}{0 0 0}
\pscustom[linewidth=0.2,linecolor=curcolor]
{
\newpath
\moveto(50.5,101)
\lineto(50.5,1)
}
}
{
\newrgbcolor{curcolor}{0 1 0}
\pscustom[linestyle=none,fillstyle=solid,fillcolor=curcolor]
{
\newpath
\moveto(76.0166592,74.34596057)
\lineto(61.8745236,88.48809617)
\lineto(61.16741682,84.95256227)
\lineto(59.04609648,92.73073685)
\lineto(66.82427106,90.60941651)
\lineto(63.28873716,89.90230973)
\lineto(77.43087276,75.76017413)
\lineto(76.0166592,74.34596057)
\closepath
}
}
{
\newrgbcolor{curcolor}{0 0 0}
\pscustom[linewidth=0.1,linecolor=curcolor]
{
\newpath
\moveto(76.0166592,74.34596057)
\lineto(61.8745236,88.48809617)
\lineto(61.16741682,84.95256227)
\lineto(59.04609648,92.73073685)
\lineto(66.82427106,90.60941651)
\lineto(63.28873716,89.90230973)
\lineto(77.43087276,75.76017413)
\lineto(76.0166592,74.34596057)
\closepath
}
}
{
\newrgbcolor{curcolor}{0 1 0}
\pscustom[linestyle=none,fillstyle=solid,fillcolor=curcolor]
{
\newpath
\moveto(77.43086368,77.17438956)
\curveto(78.21191581,76.39333743)(78.21191871,75.12700461)(77.43087015,74.34595604)
\curveto(76.64982158,73.56490748)(75.38348876,73.56491038)(74.60243663,74.34596251)
\curveto(73.82138449,75.12701465)(73.8213816,76.39334747)(74.60243016,77.17439603)
\curveto(75.38347872,77.9554446)(76.64981154,77.9554417)(77.43086368,77.17438956)
\closepath
}
}
{
\newrgbcolor{curcolor}{0 0 0}
\pscustom[linewidth=0.1,linecolor=curcolor]
{
\newpath
\moveto(77.43086368,77.17438956)
\curveto(78.21191581,76.39333743)(78.21191871,75.12700461)(77.43087015,74.34595604)
\curveto(76.64982158,73.56490748)(75.38348876,73.56491038)(74.60243663,74.34596251)
\curveto(73.82138449,75.12701465)(73.8213816,76.39334747)(74.60243016,77.17439603)
\curveto(75.38347872,77.9554446)(76.64981154,77.9554417)(77.43086368,77.17438956)
\closepath
}
}
{
\newrgbcolor{curcolor}{1 1 0}
\pscustom[linestyle=none,fillstyle=solid,fillcolor=curcolor]
{
\newpath
\moveto(24.60244795,25.63517342)
\lineto(38.74458355,39.77730902)
\lineto(35.20904965,40.4844158)
\lineto(42.98722423,42.60573614)
\lineto(40.86590389,34.82756156)
\lineto(40.15879711,38.36309546)
\lineto(26.01666151,24.22095986)
\lineto(24.60244795,25.63517342)
\closepath
}
}
{
\newrgbcolor{curcolor}{0 0 0}
\pscustom[linewidth=0.1,linecolor=curcolor]
{
\newpath
\moveto(24.60244795,25.63517342)
\lineto(38.74458355,39.77730902)
\lineto(35.20904965,40.4844158)
\lineto(42.98722423,42.60573614)
\lineto(40.86590389,34.82756156)
\lineto(40.15879711,38.36309546)
\lineto(26.01666151,24.22095986)
\lineto(24.60244795,25.63517342)
\closepath
}
}
{
\newrgbcolor{curcolor}{1 1 0}
\pscustom[linestyle=none,fillstyle=solid,fillcolor=curcolor]
{
\newpath
\moveto(27.43087695,24.22096894)
\curveto(26.64982481,23.4399168)(25.38349199,23.43991391)(24.60244343,24.22096247)
\curveto(23.82139486,25.00201103)(23.82139776,26.26834385)(24.6024499,27.04939599)
\curveto(25.38350203,27.83044813)(26.64983485,27.83045102)(27.43088342,27.04940246)
\curveto(28.21193198,26.2683539)(28.21192908,25.00202108)(27.43087695,24.22096894)
\closepath
}
}
{
\newrgbcolor{curcolor}{0 0 0}
\pscustom[linewidth=0.1,linecolor=curcolor]
{
\newpath
\moveto(27.43087695,24.22096894)
\curveto(26.64982481,23.4399168)(25.38349199,23.43991391)(24.60244343,24.22096247)
\curveto(23.82139486,25.00201103)(23.82139776,26.26834385)(24.6024499,27.04939599)
\curveto(25.38350203,27.83044813)(26.64983485,27.83045102)(27.43088342,27.04940246)
\curveto(28.21193198,26.2683539)(28.21192908,25.00202108)(27.43087695,24.22096894)
\closepath
}
}
{
\newrgbcolor{curcolor}{0 0 1}
\pscustom[linestyle=none,fillstyle=solid,fillcolor=curcolor]
{
\newpath
\moveto(76.0166592,24.34596057)
\lineto(61.8745236,38.48809617)
\lineto(61.16741682,34.95256227)
\lineto(59.04609648,42.73073685)
\lineto(66.82427106,40.60941651)
\lineto(63.28873716,39.90230973)
\lineto(77.43087276,25.76017413)
\lineto(76.0166592,24.34596057)
\closepath
}
}
{
\newrgbcolor{curcolor}{0 0 0}
\pscustom[linewidth=0.1,linecolor=curcolor]
{
\newpath
\moveto(76.0166592,24.34596057)
\lineto(61.8745236,38.48809617)
\lineto(61.16741682,34.95256227)
\lineto(59.04609648,42.73073685)
\lineto(66.82427106,40.60941651)
\lineto(63.28873716,39.90230973)
\lineto(77.43087276,25.76017413)
\lineto(76.0166592,24.34596057)
\closepath
}
}
{
\newrgbcolor{curcolor}{0 0 1}
\pscustom[linestyle=none,fillstyle=solid,fillcolor=curcolor]
{
\newpath
\moveto(77.43086368,27.17438956)
\curveto(78.21191581,26.39333743)(78.21191871,25.12700461)(77.43087015,24.34595604)
\curveto(76.64982158,23.56490748)(75.38348876,23.56491038)(74.60243663,24.34596251)
\curveto(73.82138449,25.12701465)(73.8213816,26.39334747)(74.60243016,27.17439603)
\curveto(75.38347872,27.9554446)(76.64981154,27.9554417)(77.43086368,27.17438956)
\closepath
}
}
{
\newrgbcolor{curcolor}{0 0 0}
\pscustom[linewidth=0.1,linecolor=curcolor]
{
\newpath
\moveto(77.43086368,27.17438956)
\curveto(78.21191581,26.39333743)(78.21191871,25.12700461)(77.43087015,24.34595604)
\curveto(76.64982158,23.56490748)(75.38348876,23.56491038)(74.60243663,24.34596251)
\curveto(73.82138449,25.12701465)(73.8213816,26.39334747)(74.60243016,27.17439603)
\curveto(75.38347872,27.9554446)(76.64981154,27.9554417)(77.43086368,27.17438956)
\closepath
}
}
\end{pspicture}
\end{minipage}
\hspace{0.1cm}
\begin{minipage}{0.25\textwidth}
\centering
\psset{unit=1pt}
  \input{vf_inner}
\end{minipage}
\hspace{0.1cm}
\begin{minipage}{0.25\textwidth}
\centering
\psset{unit=1pt}
  \input{vf_outer}
\end{minipage}
\hspace{0.1cm}
\begin{minipage}{0.25\textwidth}
\centering
\psset{unit=1pt}
  \input{vf_total}
\end{minipage}
\caption{From left to right: a vector field, its copy from inner rotation, outer rotation, total rotation.}
\end{figure}
Two multivector fields $\A(\x),\B(\x):\R^m\to\clifford{p,q}$ differ by an \textbf{inner rotation} if they suffice 
\begin{equation}
\begin{aligned}
\A(\x)=\B(\operatorname{R} _{-\alpha}(\x)).
\end{aligned}
\end{equation}
It can be interpreted like the starting position of every vector is rotated by $\alpha$. Then the old vector is reattached at the new position, but it still points into the old direction. The inner rotation is suitable to describe the rotation of a color image. The color is represented as a vector and does not change when the picture is turned. 
\par
Another kind of misalignment we want to mention is the \textbf{outer rotation}
\begin{equation}
\begin{aligned}
 \A(\x)=\operatorname{R} _{\alpha}(\B(\x)).
\end{aligned}
\end{equation}
Here every vector on the vector field $\A$ is the rotated copy of every vector in the vector field  $\B$. The vectors are rotated independently from their positions. This kind of rotation appears for example in color images, when the color space is turned but the picture is not moved, compare \cite{MSE03}. 
\par
The third and in this paper most relevant kind is the \textbf{total rotation}
\begin{equation}
\begin{aligned}
 \A(\x)=\operatorname{R} _{\alpha}(\B(\operatorname{R} _{-\alpha}(\x))).
\end{aligned}
\end{equation}
The positions and the multivectors are stiffly connected during this kind of rotation. If domain and codomain are of equal dimension it can be interpreted as a coordinate transform, as looking at the multivector field from another point of view. A total rotation is the most intuitive of the misalignments, it occurs in physical vector fields like for example fluid mechanics, and aerodynamics.
\par
With respect to the definition of the correlation there are different formulae in current literature, \cite{Ebl06,MSE03}. We prefer the following one because it satisfies a geometric generalization of the Wiener-Khinchin theorem and because it coincides with the definition of the standard cross-correlation for complex functions in the special case of $\clifford{0,1}$. 
%For vector fields they mostly coincide anyways because of $\overline{\v(\x)}=\v(\x)$.
%
\begin{defn}
 The \textbf{geometric cross correlation} of two multivector valued functions $\A(\x),\B(\x):\R^m\to\clifford{p,q}$ is a multivector valued function defined by
\begin{equation}
\begin{aligned}
(\A\star \B)(\x):=&\int_{\R^m}\overline{\A(\y)}\B(\y+\x)\d^m\y,
\end{aligned}
\end{equation} 
where $\overline{\A(\y)}=\sum\limits_{k=0}^n(-1)^{\frac12k(k-1)}\langle \A(\y)\rangle_k$ is the reversion.
\end{defn}
\begin{rem}
 To simplify notation we will make some conventions. Without loss of generality we assume the integrable vector fields to be normalized with respect to the $L^2$-norm. That way the normalized cross correlation coincides with its unnormalized counterpart. We will also just analyze the correlation at the origin. Since our vector fields are not shifted, the origin of coordinates is the place of the translational registration. If the vector fields should also differ by an inner shift, our methods can be applied analogously to this location.
\end{rem}

% As in \cite{MSE02} we additionally restrict ourselves to angles $\alpha\in[0,\pi]$, that means we encode the sign in the bivector $\vec P$ and deal with positive angles only.
%----------------------------------------------------------------------------------------------------------------------------
\section{Motivation}
%----------------------------------------------------------------------------------------------------------------------------
The fundamental idea for this paper stems from the correlation of a two-dimensional vector field and its copy from outer rotation
% A mathematically positive outer rotation of a two-dimensional vector field by the angle $\alpha$ takes the shape
% \begin{equation}
% \begin{aligned}
% \operatorname{R} _{\alpha}(\v(\x))=e^{-\alpha \e_{12}}\v(\x).
% \end{aligned} 
% \end{equation}
% The product of the vector field and its copy at any position $\x\in\R^m$ yields
% \begin{equation}
% \begin{aligned}\label{outerprod}
% \operatorname{R} _{\alpha}(\v(\x))\v(\x)=e^{-\alpha \e_{12}}\v(\x)\v(\x)=\v(\x)^2e^{-\alpha \e_{12}},
% \end{aligned} 
% \end{equation}
% with $\v(\x)^2=||\v(\x)||_2^2\in\R$. That means the rotation can be restored completely by rotating back with the inverse of (\ref{outerprod}). In order to get a robust method we use the geometric correlation at the origin
\begin{equation}
\begin{aligned}\label{outerprod}
(\operatorname{R} _{\alpha}(\v)\star \v)(0)
&=\int_{\R^2}\overline{\operatorname{R} _{\alpha}(\v(\x))}\v(\x)\d^2\x
%\\&=\int_{\R^2}\overline{\v(\x)}e^{\alpha \e_{12}}\v(\x)\d^2\x
\\&=\int_{\R^2}\overline{e^{-\alpha \e_{12}}\v(\x)}\v(\x)\d^2\x 
\\&=||\v(\x)||_{L^2}^2e^{-\alpha \e_{12}}.
\end{aligned} 
\end{equation}
Since $||\v(\x)||_{L^2}^2\in\R$ the alignment can be restored by rotating back $\operatorname{R} _{\alpha}(\v)$ by the angle encoded in the argument.
%which evaluates all available information of the vector fields.
\par
We want to develop this idea further to analyze total rotations. In $\clifford{2,0}$ they take the shape
\begin{equation}
\begin{aligned}\label{total}
\u(\x)=\operatorname{R} _{\alpha}(\v(\operatorname{R} _{-\alpha}(\x)))=e^{-\alpha \e_{12}}\v(e^{\alpha \e_{12}}\x),
\end{aligned} 
\end{equation}
so it is not possible to predict the rotation that is encoded in the geometric correlation
% \begin{equation}
% \begin{aligned}
% \u(\x)\v(\x)=e^{-\alpha \e_{12}}\v(e^{\alpha \e_{12}}\x)\v(\x)
% \end{aligned} 
% \end{equation}
without knowing the shape of $\v$. 
\par
Vector fields that depend only on the magnitude of $\x$ are  invariant with respect to inner rotations. It is easy to see that in this case the correlation takes the same shape as in (\ref{outerprod}) and that the misalignment can be corrected applying a rotation by the angle in the argument, too. But in general the vector fields and the rotor can not be separated from the integral of the correlation
\begin{equation}
\begin{aligned}
(\u\star \v)(0)
&=\int_{\R^m}\overline{\operatorname{R} _{\alpha}(\v(\x))}\v(\x)\d^m\x
\\&=e^{-\alpha \e_{12}}\int_{\R^m}\v(e^{\alpha \e_{12}}\x)\v(\x)\d^m\x.
\end{aligned} 
\end{equation}
We dealt with a similar problem in \cite{BSH12a} when we treated the three-dimensional outer rotation. 
% \begin{equation}
% \begin{aligned}
% (\operatorname{R} _{\alpha}(\v)\star \v)(0)
% &=e^{-\frac\alpha2 \e_{12}}\int_{\R^m}\v(\x)e^{\frac\alpha2 \e_{12}}\v(\x)\d^m\x.
% \end{aligned} 
% \end{equation}
For this case we could prove that the encoded rotation is at least a fair approximation to the one sought after and that iterative application leads to the detection of the misalignment sought after. 
\par
Trying to adapt this idea to total rotations we discovered that this result does not apply to all two-dimensional vector fields, compare the following counterexample.

\begin{figure}[ht]
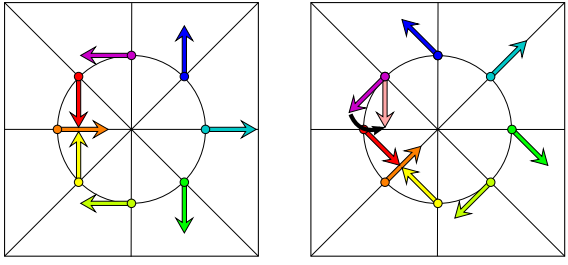
\label{f:gegenbsp}
% \begin{minipage}{0.25\textwidth}
% \centering
% %\psset{unit=1pt}
%   \includegraphics[width=\textwidth]{gegenbsp}
% \end{minipage}
% \hspace{0.1cm}
\begin{minipage}{0.23\textwidth}
\centering
\psset{unit=0.8pt}
  \input{gegenbsp}
\end{minipage}
\hspace{0.1cm}
% \begin{minipage}{0.25\textwidth}
% \centering
% %\psset{unit=1pt}
%   \includegraphics[width=\textwidth]{gegenbsp0,25pi}
% \end{minipage}
% \hspace{0.1cm}
\begin{minipage}{0.23\textwidth}
\centering
\psset{unit=0.8pt}
  \input{gegenbsp0,25pi}
\end{minipage}
\caption{Left: vector field from the counter example. Right: math. pos. rotated copy by $\frac\pi4$. At each position the same rotor, depicted as black arrow, is contributed to the correlation.}
\end{figure}
\begin{ex}
Let $\v:B_1(0)\to\R^2$ be the vector field from Figure \ref{f:gegenbsp} vanishing outside the unit circle take the shape 
\begin{equation}
\begin{aligned}
\v(r,\varphi)=&\e_1e^{2\varphi \e_{12}}%=\cos(2\varphi)\e_1+\sin(2\varphi)\e_2,
\end{aligned}
\end{equation}
% \begin{equation}
% \begin{aligned}
% \v(r,\varphi)=&\begin{cases}
%          \e_1e^{2\varphi \e_{12}}%=\cos(2\varphi)\e_1+\sin(2\varphi)\e_2,
% &\text{ for }r<1\\
% 0,&\text{ else. }
%         \end{cases}
% \end{aligned}
% \end{equation}
expressed in polar coordinates. Then its rotated copy suffices
\begin{equation}
\begin{aligned}
\u(r,\varphi)=&\e_1e^{(2\varphi-\alpha) \e_{12}}%= \sin(\alpha) \sin(2\varphi)+\cos(\alpha)\cos(2\varphi)\e_1+( \cos(\alpha)\sin(2\varphi)+\sin(\alpha)\cos(2\varphi))\e_2
\end{aligned}
\end{equation}
% \begin{equation}
% \begin{aligned}
% \u(r,\varphi)=&\begin{cases}
%          \e_1e^{(2\varphi-\alpha) \e_{12}}%= \sin(\alpha) \sin(2\varphi)+\cos(\alpha)\cos(2\varphi)\e_1+( \cos(\alpha)\sin(2\varphi)+\sin(\alpha)\cos(2\varphi))\e_2
% ,&\text{ for }r<1\\
% 0,&\text{ else }
%         \end{cases}
% \end{aligned}
% \end{equation}
inside the unit circle and the correlation of the two is
\begin{equation}
\begin{aligned}
(\u(\x)\star\v(\x))(0)
=&\int_{B_1(0)}\e_1e^{(2\varphi-\alpha) \e_{12}}\e_1e^{2\varphi \e_{12}}r\d r\d\varphi
\\=&\int_{B_1(0)}\e_1\e_1e^{-(2\varphi-\alpha) \e_{12}}e^{2\varphi \e_{12}}r\d r\d\varphi
\\=&\int_{B_1(0)}e^{\alpha \e_{12}}r\d r\d\varphi
\\=&\pi e^{\alpha \e_{12}}.
\end{aligned}
\end{equation}
If we want to correct the misalignment by rotating back with its inverse like in (\ref{outerprod}) we would rotate in the completely wrong direction and double the misalignment with each step, because no matter how the rotational misalignment was, we always detect its negative. So imagine starting the iterative algorithm from \cite{BSH12a} with $\alpha=\frac{2\pi}3$. It would become periodic $\frac{2\pi}3,\frac{4\pi}3,\frac{8\pi}3=\frac{2\pi}3,...$ and not converge at all.
\end{ex}
But the idea applies to all linear fields. We will show in the next sections, that iteratively rotating back with the inverse of the normalized geometric correlation will detect the correct misalignment of any two-dimensional linear vector field and its copy from total rotation.
%----------------------------------------------------------------------------------------------------------------------------
\section{Easy Linear Examples}
%----------------------------------------------------------------------------------------------------------------------------
Assume a linear vector field in two dimensions
\begin{equation}
\begin{aligned}
\v(\x)=(a_{11}x_1+a_{12}x_2)\e_1+(a_{21}x_1+a_{22}x_2)\e_2
\end{aligned} 
\end{equation}
with real coefficients.  
% %----------------------------------------------------------------------------------------------------------------------------
% \section{Product after Total Rotation}
% %----------------------------------------------------------------------------------------------------------------------------
% A total rotation takes the shape
% \begin{equation}
% \begin{aligned}
% D(\v(\operatorname{R} _{-\alpha}(\x)))=&e^{-\alpha \e_{12}}(\v(e^{\alpha \e_{12}}\x))
% % %\\=&e^{-\alpha \e_{12}}(\a(e^{\alpha \e_{12}}(x_1\e_1+x_2\e_2)))
% % \\=&e^{-\alpha \e_{12}}(\v((\cos(\alpha)x_1+\sin(\alpha)x_2)\e_1+(-\sin(\alpha)x_1+\cos(\alpha)x_2)\e_2))
% % \\=&e^{-\alpha \e_{12}}((a_{11}(\cos(\alpha)x_1+\sin(\alpha)x_2)+a_{12}(-\sin(\alpha)x_1+\cos(\alpha)x_2)\e_2))\e_1
% % \\&+(a_{21}(\cos(\alpha)x_1+\sin(\alpha)x_2)+a_{22}(-\sin(\alpha)x_1+\cos(\alpha)x_2)\e_2))\e_2)
% % \\=&e^{-\alpha \e_{12}}((a_{11}(\cos(\alpha)x_1+\sin(\alpha)x_2)+a_{12}(-\sin(\alpha)x_1+\cos(\alpha)x_2)\e_2))\e_1
% % \\&+(a_{21}(\cos(\alpha)x_1+\sin(\alpha)x_2)+a_{22}(-\sin(\alpha)x_1+\cos(\alpha)x_2)\e_2))\e_2)
% \end{aligned} 
% \end{equation}
% so the product 
% \begin{equation}
% \begin{aligned}
% D(\v(\operatorname{R} _{-\alpha}(\x)))\v(\x)
% % =&e^{-\alpha \e_{12}}((a_{11}(\cos(\alpha)x_1+\sin(\alpha)x_2)+a_{12}(-\sin(\alpha)x_1+\cos(\alpha)x_2)\e_2))\e_1
% % \\&+(a_{21}(\cos(\alpha)x_1+\sin(\alpha)x_2)+a_{22}(-\sin(\alpha)x_1+\cos(\alpha)x_2)\e_2))\e_2)
% % \\&((a_{11}x_1+a_{12}x_2)\e_1+(a_{21}x_1+a_{22}x_2)\e_2)
% \end{aligned} 
% \end{equation}
% is rather complicated. 
Before analyzing the general linear case, let us look the examples in Figure \ref{f:2}, the saddles 
\begin{equation}
\begin{aligned}
\vec a(\x)=&x_1\e_1-x_2\e_2,\\
\vec b(\x)=&x_2\e_1+x_1\e_2,\\
\end{aligned} 
\end{equation}
the source 
\begin{equation}
\begin{aligned}
\vec c(\x)=&x_1\e_1+x_2\e_2,
\end{aligned} 
\end{equation}
and the vortex 
\begin{equation}
\begin{aligned}
\vec d(\x)=&x_2\e_1-x_1\e_2.
\end{aligned} 
\end{equation}

\begin{figure}[ht]\label{f:2}
\begin{minipage}{0.25\textwidth}
\centering
%\psset{unit=1pt}
  \includegraphics[width=\textwidth]{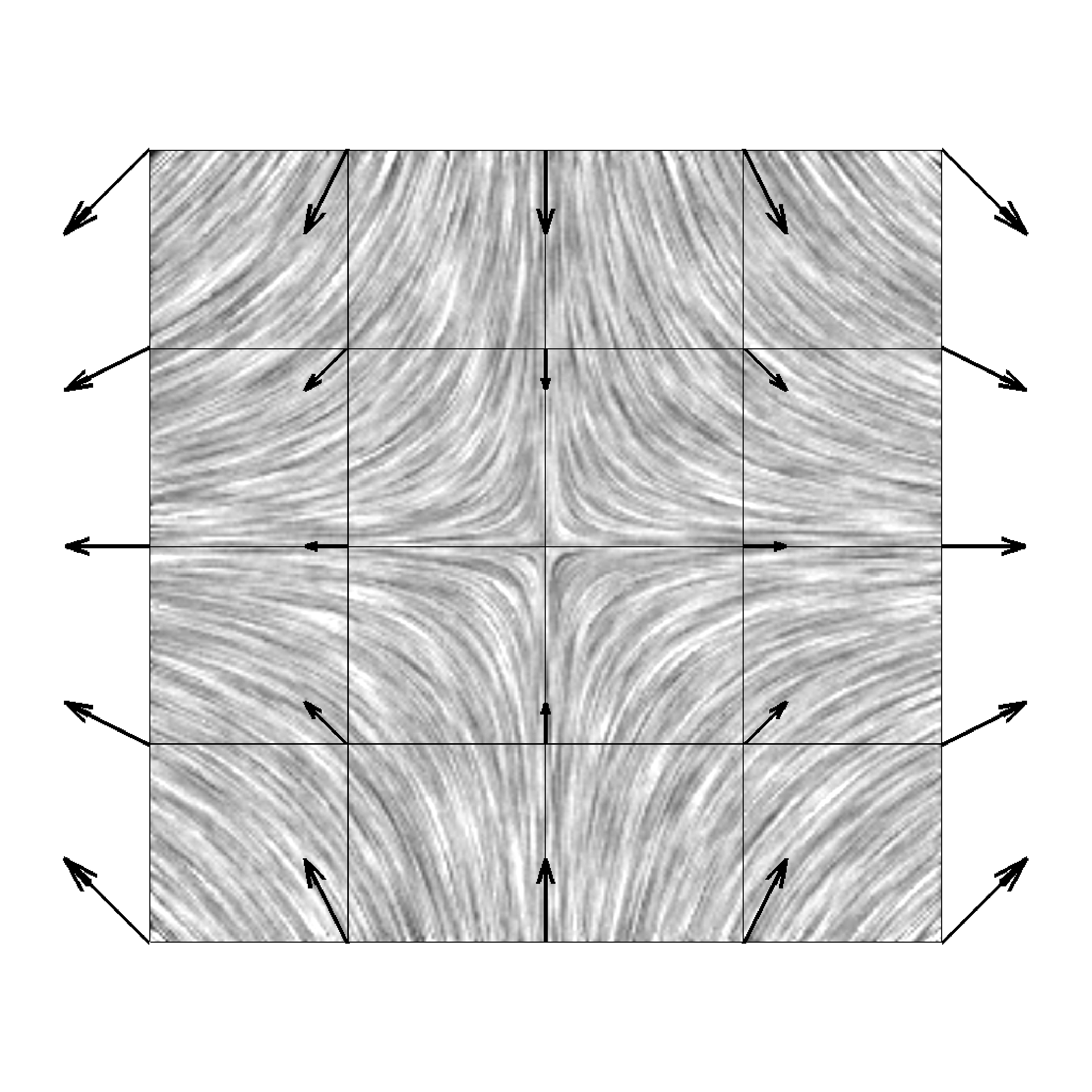}
\end{minipage}
\hspace{0.1cm}
\begin{minipage}{0.25\textwidth}
\centering
%\psset{unit=1pt}
  \includegraphics[width=\textwidth]{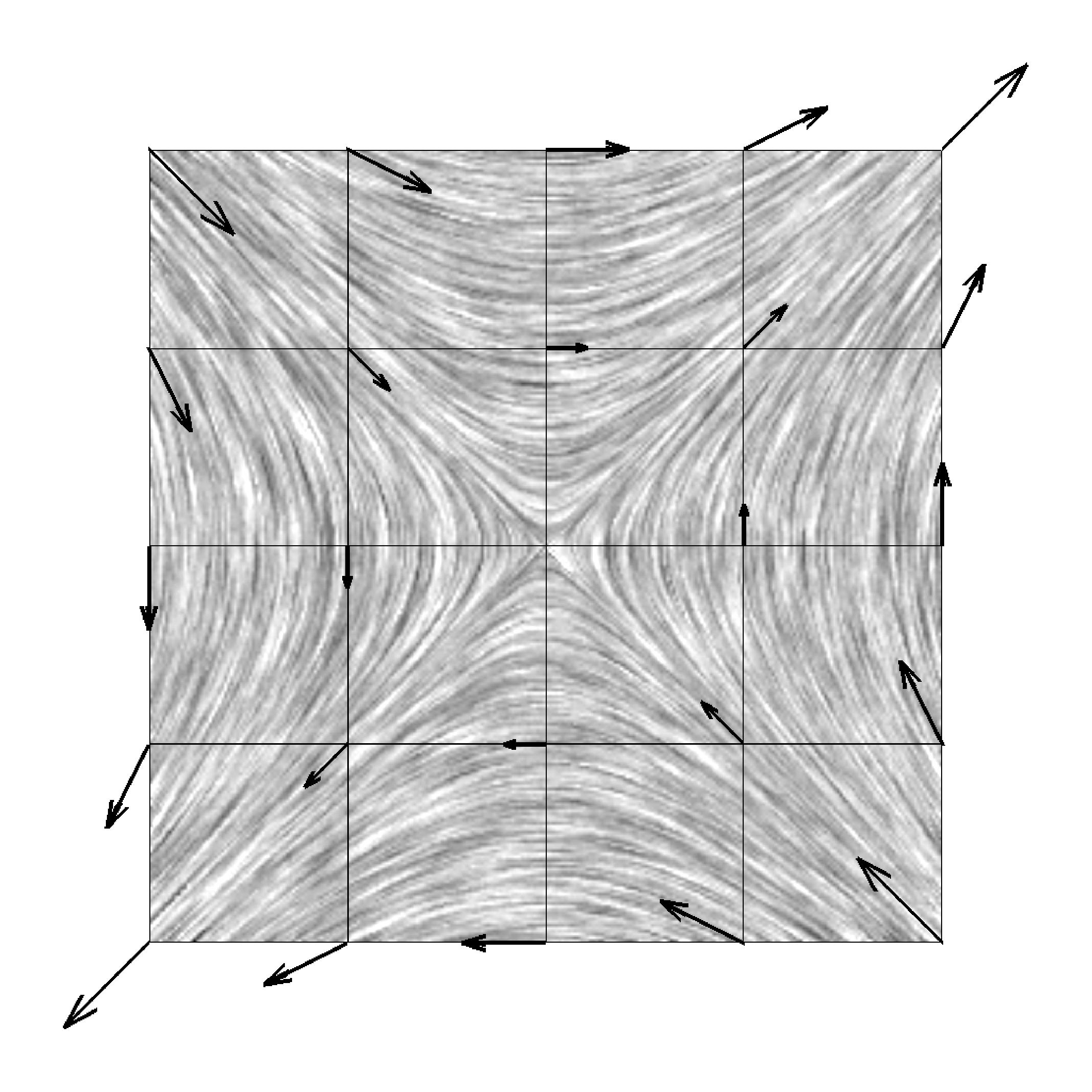}
\end{minipage}
\hspace{0.1cm}
\begin{minipage}{0.25\textwidth}
\centering
%\psset{unit=1pt}
  \includegraphics[width=\textwidth]{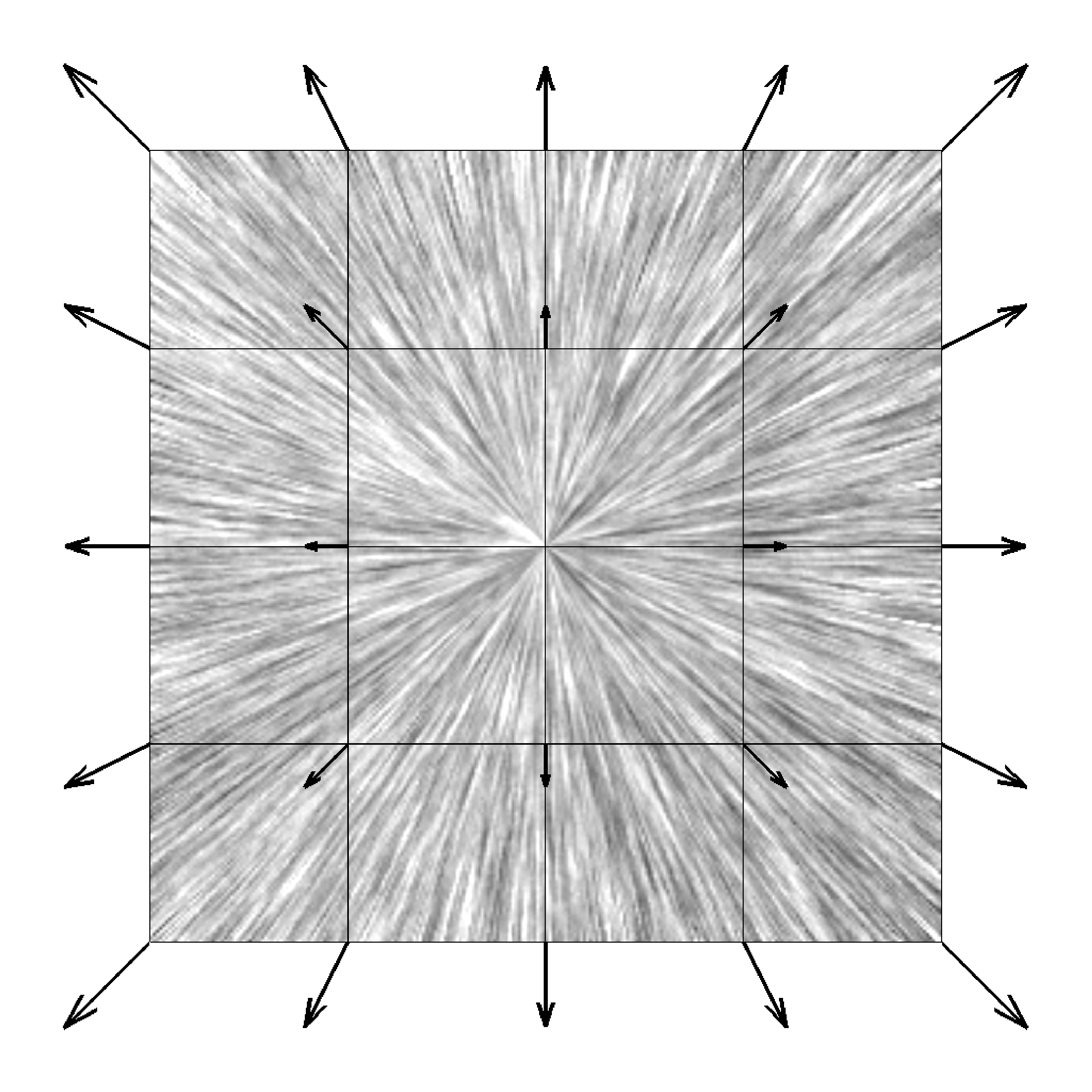}
\end{minipage}
\hspace{0.1cm}
\begin{minipage}{0.25\textwidth}
\centering
%\psset{unit=1pt}
  \includegraphics[width=\textwidth]{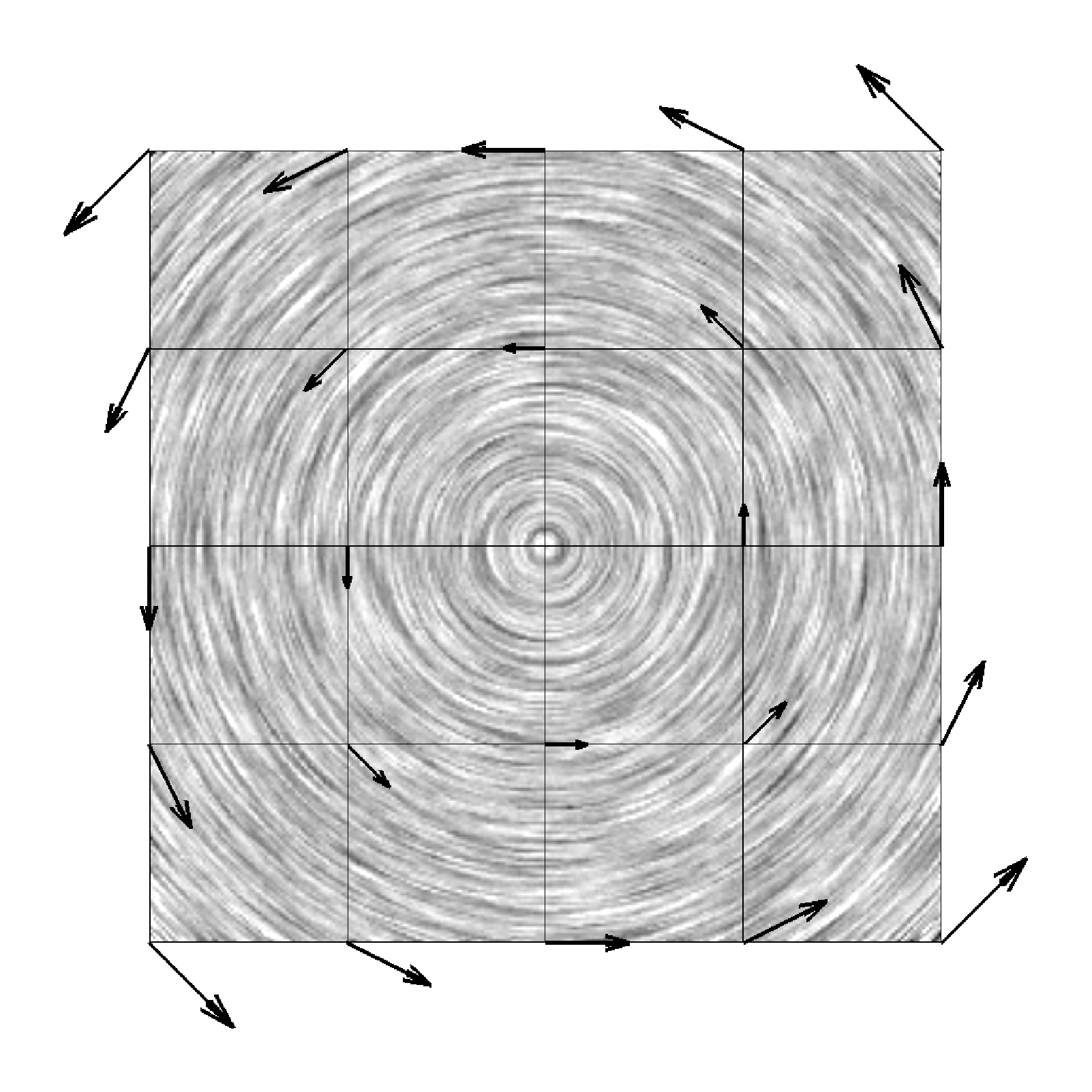}
\end{minipage}
\caption{From left to right: saddles $\a(\x)$, $\b(\x)$, source $\c(\x)$, and vortex $\vec d(\x)$ visualized with hedgehogs and LIC \cite{CL93}.}
\end{figure}
\begin{rem}\label{r:alt_parts} 
Instead of using the coefficients the vector fields from Figure \ref{f:2} can analogously be expressed by basic transformations
\begin{equation}
\begin{aligned}
\c(\x)=&\x,\\
\a(\x)=&-\e_2\x\e_2,\\
\vec d(\x)=&\e_{12}\x,\\
\b(\x)=&-\e_2\vec d(\x)\e_2.
\end{aligned} 
\end{equation}
The first three of them are the identity, a reflection at the hyperplane perpendicular to $\e_2$ and a rotation about $-\frac\pi2$. The last one $\b(\x)=-\e_2\vec d(\x)\e_2$ can be seen as a rotation about $-\frac\pi2$ followed by a reflection at the hyperplane perpendicular to $\e_2$ or as $\b(\x)=-\e_{12}\a(\x)$ a reflection at the hyperplane perpendicular to $\e_2$ followed by a rotation about $\frac\pi2$ or alternatively as $\b(\x)=-1/2(\e_1-\e_2)\x(\e_1-\e_2)$ just a reflection at the hyperplane perpendicular to $\e_1-\e_2$. From this description we immediately get
\begin{equation}
\begin{aligned}
\a(\x)\perp \b(\x),\\
\c(\x)\perp \vec d(\x),
\end{aligned} 
\end{equation}
and
\begin{equation}
\begin{aligned}
\vec{a}(\vec{x})^2= \vec{b}(\vec{x})^2=\vec{c}(\vec{x})^2= \vec{d}(\vec{x})^2=\vec{x}^2.
\end{aligned} 
\end{equation}
\end{rem}
The geometric products of them and their rotated copies at any position take very simple forms.
\begin{ex}\label{b:saddle}
 For a saddle $\a(\x)=x_1\e_1-x_2\e_2=-\e_2\x\e_2$ we get
\begin{equation}
\begin{aligned}
\operatorname{R} _{\alpha}(\a(\operatorname{R} _{-\alpha}(\x)))
% %=&%\operatorname{R} _{\alpha}(\a(e^{\alpha \e_{12}}\x))
% %\\=&\operatorname{R} _{\alpha}(\a(e^{\alpha \e_{12}}(x_1\e_1-x_2\e_2)))
% =&\operatorname{R} _{\alpha}(\a((\cos(\alpha)x_1+\sin(\alpha)x_2)\e_1+(-\sin(\alpha)x_1+\cos(\alpha)x_2)\e_2))
% \\=&\operatorname{R} _{\alpha}((\cos(\alpha)x_1+\sin(\alpha)x_2)\e_1+(\sin(\alpha)x_1-\cos(\alpha)x_2)\e_2)
% \\=&e^{-\alpha \e_{12}}(x_1(\cos(\alpha)\e_1+\sin(\alpha)\e_2)+x_2(\sin(\alpha)\e_1-\cos(\alpha)\e_2))
% \\=&e^{-\alpha \e_{12}}(x_1 e^{-\alpha \e_{12}}\e_1+x_2 e^{-\alpha \e_{12}}\e_2  )
=&e^{-\alpha \e_{12}}(-\e_2e^{\alpha \e_{12}}\x\e_2)
\\=&e^{-2\alpha \e_{12}}(-\e_2\x\e_2)
\\=&e^{-2\alpha \e_{12}}\a(\x)
\end{aligned} 
\end{equation}
and therefore the product 
\begin{equation}
\begin{aligned}\label{prod_s1}
\operatorname{R} _{\alpha}(\a(\operatorname{R} _{-\alpha}(\x)))\a(\x)=&\a(\x)^2e^{-2\alpha \e_{12}}
\end{aligned} 
\end{equation}
reveals twice the angle we looked for.
\end{ex}
\begin{rem}\label{r:pi/2}
The argument of $\a(\x)^2e^{-2\alpha \e_{12}}$ is only $-2\alpha$ for $\alpha\in[-\frac\pi2,\frac\pi2]$. 
% Since we work with angles of magnitude less or equal than $\pi$ the argument takes the shape
% \begin{equation}
% \begin{aligned}
% \arg(\a(\x)^2e^{-2\alpha \e_{12}})=\begin{cases}-2\alpha & \text{ for }\alpha\in[-\pi/2,\pi/2],\\
% -2\alpha+2\pi & \text{ for }\alpha\in(\pi/2,\pi],\\
% -2\alpha-2\pi & \text{ for }\alpha\in[-\pi,-\pi/2).\end{cases}
% \end{aligned} 
% \end{equation}
Keeping track of the case differentiation is a hassle, so we restrict ourselves to this interval during the whole paper. For two-dimensional linear vector fields this is no restriction, because the by $\pi$ totally rotated copy of $\v(\x)$ suffices
\begin{equation}
\begin{aligned}
\operatorname{R}_\pi(\v(\operatorname{R}_{-\pi}(\x)))=&e^{-\pi \e_{12}}(\v(e^{\pi \e_{12}}(\x)))
\\=&-\v(-\x)
\\=&\v(\x).
\end{aligned} 
\end{equation}
That means an angle $\alpha$ in $[-\frac\pi2,\frac\pi2]$ can always be found to describe the rotational difference of $\v(\x)$ and its rotated copy.
\end{rem}
\begin{ex}\label{b:vortex}
For a vortex $\vec d(\x)=x_2\e_1-x_1\e_2=\e_{12}\x$ we get
\begin{equation}
\begin{aligned}
\operatorname{R} _{\alpha}(\vec d(\operatorname{R} _{-\alpha}(\x)))
% %=&\operatorname{R} _{\alpha}(\b(e^{\alpha \e_{12}}\x))
% =&\operatorname{R} _{\alpha}(\vec d((\cos(\alpha)x_1+\sin(\alpha)x_2)\e_1+(-\sin(\alpha)x_1+\cos(\alpha)x_2)\e_2))
% \\=&e^{-\alpha \e_{12}}((-\sin(\alpha)x_1+\cos(\alpha)x_2)\e_1-(\cos(\alpha)x_1+\sin(\alpha)x_2)\e_2)
% \\=&e^{-\alpha \e_{12}}(x_1(-\cos(\alpha)\e_2-\sin(\alpha)\e_1)+x_2(-\sin(\alpha)\e_2+\cos(\alpha)\e_1))
% \\=&e^{-\alpha \e_{12}}(-x_1 e^{\alpha \e_{12}}\e_2+x_2 e^{\alpha \e_{12}}\e_1  )
=&e^{-\alpha \e_{12}}\e_{12}e^{\alpha \e_{12}}\x
\\=&e^{-\alpha \e_{12}}e^{\alpha \e_{12}}\e_{12}\x
\\=&\vec d(\x)
\end{aligned} 
\end{equation}
and therefore the product
\begin{equation}
\begin{aligned}\label{prod_v}
\operatorname{R} _{\alpha}(\vec d(\operatorname{R} _{-\alpha}(\x)))\vec d(\x)=&\vec d(x)^2
\end{aligned} 
\end{equation}
is real valued, what is to be interpreted as an angle of zero. That is not a surprise, because a vortex is invariant with respect to total rotations in the plane. Therefore it is not disturbing that we do not get the rotational information either. 
\end{ex}
\begin{ex}%{Saddle, Source}
A saddle with the shape $\b(\x)=x_2\e_1+x_1\e_2=\e_1\x\e_2$ suffices
\begin{equation}
\begin{aligned}\label{s2}
\operatorname{R} _{\alpha}(\vec b(\operatorname{R} _{-\alpha}(\x)))=&e^{-\alpha \e_{12}}\e_1e^{\alpha \e_{12}}\x\e_2
\\=&e^{-2\alpha \e_{12}}\e_1\x\e_2
\\=&e^{-2\alpha \e_{12}}\b(\x)
\end{aligned} 
\end{equation}
So the product of the saddles always leads to $-2\alpha$, too.% as in Example \ref{b:saddle}.
\end{ex}
\begin{ex}
A source $\c(\x)=x_1\e_1+x_2\e_2=\x$ leads to
\begin{equation}
\begin{aligned}\label{s}
\operatorname{R} _{\alpha}(\vec c(\operatorname{R} _{-\alpha}(\x)))=&e^{-\alpha \e_{12}}e^{\alpha \e_{12}}\x
\\=&\vec \x
\\=&\vec c(\x).
\end{aligned} 
\end{equation}
So the product of sources also leads to zero angle.% like for the vortex in \ref{b:vortex}.
\end{ex}
%----------------------------------------------------------------------------------------------------------------------------
\section{Composition of linear Fields from the examples}
%----------------------------------------------------------------------------------------------------------------------------
\begin{lem}\label{l:vf_split}
Any linear vector field can be expressed as a linear combination of the four examples from the preceding section. That means for %the saddles 
\begin{equation}
\begin{aligned}
\vec a(\x)=&x_1\e_1-x_2\e_2,\\
\vec b(\x)=&x_2\e_1+x_1\e_2,\\
% \end{aligned} 
% \end{equation}
% %$\vec a(\x)=x_1\e_1-x_2\e_2$, $\vec b(\x)=x_2\e_1+x_1\e_2$,
% the source and the vortex %$\vec c(\x)=x_1\e_1+x_2\e_2$ and the vortex $\vec d(\x)=x_2\e_1-x_1\e_2$
% \begin{equation}
%\begin{aligned}
\vec c(\x)=&x_1\e_1+x_2\e_2,\\
\vec d(\x)=&x_2\e_1-x_1\e_2\\
\end{aligned} 
\end{equation}
 there are $a,b,c,d\in\R$, such that
\begin{equation}
\begin{aligned}
\v(x)=a\vec a(\x)+b\vec b(\x)+c\vec c(\x)+d\vec d(\x).
\end{aligned} 
\end{equation}
\end{lem} 
\begin{proof}
For the proof set 
\begin{equation}
\begin{aligned}
a=&\frac12(a_{11}-a_{22}),\\
b=&\frac12(a_{12}+a_{21}),\\
c=&\frac12(a_{11}+a_{22}),\\
d=&\frac12(-a_{12}+a_{21}),
\end{aligned} 
\end{equation}
that leads to
\begin{equation}
\begin{aligned}
a_{11}=&a+c,\\
a_{12}=&b-d,\\
a_{21}=&b+d,\\
a_{22}=&-a+c,
\end{aligned} 
\end{equation}
and therefore
\begin{equation}
\begin{aligned}
\v(x)=&a_{11}x_1\e_1+a_{12}x_1\e_2+a_{21}x_2\e_1+a_{22}x_2\e_2 
\\=&(a+c)x_1\e_1+(b-d)x_1\e_2
\\&+(b+d)x_2\e_1+(-a+c)x_2\e_2
\\=&a(x_1\e_1-x_2\e_2)+b(x_2\e_1+x_1\e_2)
\\&+c(x_1\e_1+x_2\e_2)+d(x_2\e_1-x_1\e_2)
\\=&a\vec a(\x)+b\vec b(\x)+c\vec c(\x)+d\vec d(\x).
\end{aligned} 
\end{equation}
\end{proof}
\begin{rem}
 The decomposition from Lemma \ref{l:vf_split} is analogous to the description of two-dimensional linear vector fields by Scheuermann in \cite{Sch99}. He stresses the isomorphism of the vectors and the rotors in $\clifford{2,0}$, denoting 
\begin{equation}
\begin{aligned}\label{E(z)}
\v(x)=&\bar E(\vec r)
\\=&E(z,\bar z)\e_1
\\=&((\alpha+\e_{12}\beta)z+(\gamma+\e_{12}\delta)\bar z)\e_1
\end{aligned} 
\end{equation}
with $z=x_1+x_2\e_{12},\bar z=x_1-x_2\e_{12}$. It leads to
\begin{equation}
\begin{aligned}
\v(x)=&((\alpha+\e_{12}\beta)(x_1+x_2\e_{12})
\\&+(\gamma+\e_{12}\delta)(x_1-x_2\e_{12}))\e_1
\\=&\alpha(x_1\e_1-x_2\e_2)+\beta(-x_2\e_1-x_1\e_2)
\\&+\gamma(x_1\e_1+x_2\e_2)+\delta(x_2\e_1-x_1\e_2)
\\=&\alpha\a(\x)-\beta\b(\x)+\gamma\c(\x)+\delta\vec d(\x),
\end{aligned} 
\end{equation}
which is the same decomposition as above if we identify $\alpha=a,-\beta=b,\gamma=c$ and $\delta=d$.
\end{rem}
%
%----------------------------------------------------------------------------------------------------------------------------
\section{Influence of Superposition on the Product}
%----------------------------------------------------------------------------------------------------------------------------
\begin{lem}\label{l:v1,v2}
Let the part in Lemma \ref{l:vf_split} of the two-dimensional linear vector field $\v(\x)$ consisting of the two saddles be denoted by 
\begin{equation}
\begin{aligned}
\v_1(\x)=a\a(\x)+b \b(\x)=(a-b\e_{12})(-\e_2\x\e_2)
\end{aligned} 
\end{equation}
 and the part consisting of the vortex and the source by 
\begin{equation}
\begin{aligned}
\v_2(\x)=c\c(\x)+d \vec d(\x)=(c+d\e_{12})\x.
\end{aligned} 
\end{equation}
Then their totally rotated copies take the shapes
\begin{equation}
\begin{aligned}
\operatorname{R} _{\alpha}(\v_1(\operatorname{R} _{-\alpha}(\x)))=&e^{-2\alpha \e_{12}}\v_1(\x),\\
% \end{aligned} 
% \end{equation}
% and 
% \begin{equation}
% \begin{aligned}
\operatorname{R} _{\alpha}(\v_2(\operatorname{R} _{-\alpha}(\x)))=&\v_2(\x).
\end{aligned} 
\end{equation}
\end{lem}
\begin{proof}
The assertion follows from the linearity of the rotation and the calculations from (\ref{prod_s1}) and (\ref{prod_v}).
%  \begin{equation}
% \begin{aligned}
% \operatorname{R} _{\alpha}(\v_1(\operatorname{R} _{-\alpha}(\x)))
% =&a\operatorname{R} _{\alpha}(\a(\operatorname{R} _{-\alpha}(\x)))+b\operatorname{R} _{\alpha}(\b(\operatorname{R} _{-\alpha}(\x)))
% \\=&ae^{-2\alpha \e_{12}}\a(\x)+ be^{-2\alpha \e_{12}}\b(\x)
% \\=&e^{-2\alpha \e_{12}}(a\a(\x)+b\b(\x))
% \\=&e^{-2\alpha \e_{12}}\v_1(\x)
% \end{aligned} 
% \end{equation}
\end{proof}
%
% \begin{lem}\label{l:v2}
%  The totally rotated copy of any linear combination of the source and the vortex $\v_2(\x)=c\c(\x)+d \vec d(\x)=(c+d\e_{12})\x$ takes the shape
% \begin{equation}
% \begin{aligned}
% \operatorname{R} _{\alpha}(\v_2(\operatorname{R} _{-\alpha}(\x)))=&\v_2(\x).
% \end{aligned} 
% \end{equation}
% \end{lem}
% \begin{proof}
% The assertion follows from the linearity of the rotation (\ref{s}) and (\ref{prod_v}).
% % \begin{equation}
% % \begin{aligned}
% % \operatorname{R} _{\alpha}(\v_2(\operatorname{R} _{-\alpha}(\x)))=&c\operatorname{R} _{\alpha}(\c(\operatorname{R} _{-\alpha}(\x)))+ d\operatorname{R} _{\alpha}(\vec d(\operatorname{R} _{-\alpha}(\x)))
% % \\=&c\c(\x)+d \vec d(\x)
% % \\=&\v_2(\x)
% % \end{aligned} 
% % \end{equation} 
% \end{proof}
%So the angle that results from the product of any linear combination of these vector fields is the same as the product of the components.
%
\begin{rem}
By means of the description by Scheuermann in \cite{Sch99} this result takes the following shape: 
%We produce the copy of a linear vector field from total rotation by $\alpha$. 
The product of a two-dimensional linear vector field and its copy from total rotation by $\alpha$ yields an argument of $-2\alpha$ iff the vector field $E(z,\bar z)$, compare (\ref{E(z)}), only depends on $z$ and an argument of zero, iff it only depends on $\bar z$.
\end{rem}
%
%Now we consider general linear vector fields $\v(\x)$. Analogously to Lemma \ref{l:vf_split} it can be expressed as the superposition of $\vec{v}_1(\vec{x})$ of Lemma \ref{l:v1} and $\vec{v}_2(\vec{x})$ of Lemma \ref{l:v2}. 
%
\begin{lem}\label{l:product}
Let $\v_1(\x)=(a-b\e_{12})(-\e_2\x\e_2),\v_2(\x)=(c+d\e_{12})\x$ be the fields from Lemma \ref{l:v1,v2}. The product of any two-dimensional linear vector field $\v(\x)$ and its totally rotated copy $\u(\x)=\operatorname{R} _{\alpha}(\v(\operatorname{R} _{-\alpha}(\x)))$ takes the shape
\begin{equation}
\begin{aligned}\label{assertion}
\u(\x)\v(\x)=&e^{-2\alpha \e_{12}}\v_1(\x)^2+e^{-2\alpha \e_{12}}\v_1(\x)\v_2(\x)
\\&+\v_2(\x)\v_1(\x)+\v_2(\x)^2
\end{aligned} 
\end{equation}
with 
% with $\v_1(\x)=a\a(\x)+b \b(\x)$ %=(((a_{11}-a_{22})x_1+(a_{12}+a_{21}) x_2)\e_1+((a_{22}-a_{11})x_2+(a_{12}+a_{21}) x_1)\e_2$ from Lemma \ref{l:v1} 
% and $\v_2(\x)=c\c(\x)+d \vec d(\x)$ %=((a_{11}+a_{22})x_1+(-a_{12}+a_{22}) x_2)\e_1+((a_{11}+a_{22})x_2+(a_{12}-a_{22}) x_1)\e_2$ from Lemma \ref{l:v2}.
% from Lemma \ref{l:vf_split}, in detail
\begin{equation}
\begin{aligned}\label{detail}
\v_1(\x)^2=&(a^2+b^2)(x_1^2+x_2^2),\\
\v_1(\x)\v_2(\x)=&(a-b \e_{12})(c-d\e_{12}) (x_1^2-x_2^2+2x_1x_2 \e_{12}),\\
%(ac+bd)(x_1^2-x_2^2)+2(bc-ad)x_1x_2+((ad-bc)(x_1^2-x_2^2)+2(ac+bd)x_1x_2)\e_{12},\\
%\v_2(\x)\v_1(\x)=&(ac+bd)(x_1^2-x_2^2)+2(bc-ad)x_1x_2-((ad-bc)(x_1^2-x_2^2)+2(ac+bd)x_1x_2)\e_{12},\\
\v_2(\x)\v_1(\x)=&(a+b \e_{12})(c+d\e_{12}) (x_1^2-x_2^2-2x_1x_2 \e_{12}),\\
\v_2(\x)^2=&(c^2+d^2)(x_1^2+x_2^2).
\end{aligned} 
\end{equation}
\end{lem}
 \begin{proof}
  We have seen in Lemma \ref{l:vf_split} that the vector field can be split into $\v(x)=\v_1(\x)+\v_2(\x)$, with $\v_1(\x)=a\a(\x)+b\b(\x)$ and $\v_2(\x)=c\c(\x)+d\vec d(\x)$. Applying Lemma \ref{l:v1,v2} leads to 
\begin{equation}
\begin{aligned}
\operatorname{R} _{\alpha}(\v(\operatorname{R} _{-\alpha}(\x)))=&\operatorname{R} _{\alpha}((\v_1+\v_2)(\operatorname{R} _{-\alpha}(\x)))
\\=&\operatorname{R} _{\alpha}(\v_1(\operatorname{R} _{-\alpha}(\x)))+ \operatorname{R} _{\alpha}(\v_2(\operatorname{R} _{-\alpha}(\x)))
\\=&e^{-2\alpha \e_{12}}\v_1(\x)+\v_2(\x)
\end{aligned} 
\end{equation}
and therefore the product suffices (\ref{assertion}). 
% \begin{equation}
% \begin{aligned}
% \operatorname{R} _{\alpha}(\v(\operatorname{R} _{-\alpha}(\x)))\v(\x)=&e^{-2\alpha \e_{12}}\v_1(\x)^2+e^{-2\alpha \e_{12}}\v_1(\x)\v_2(\x)
% \\&+\v_2(\x)\v_1(\x)+\v_2(\x)^2.
% \end{aligned} 
% \end{equation}
The assertions about the exact shape of the summands follow from straight calculation, we only give the derivation of one of the mixed parts representatively 
% From Remark \ref{r:alt_parts} follows
% \begin{equation}
% \begin{aligned}
% \v_1(\x)=&a\a(\x)+b\b(\x)=a\a(\x)-b\e_{12}\a(\x)=(a-b\e_{12})\a(\x),\\
% \v_2(\x)=&c\c(\x)+d\vec d(\x)=c\c(\x)+d\e{12}\vec c(\x)=(c+d\e_{12})\c(\x),
% \end{aligned} 
% \end{equation}
% so
\begin{equation}
\begin{aligned}
\v_1(\x)\v_2(\x)=&(a-b\e_{12})\a(\x)(c+d\e_{12})\c(\x)\\
=&(a-b\e_{12})(c-d\e_{12})(x_1\e_1-x_2\e_2)\x
\\=&(a-b\e_{12})(c-d\e_{12})(x_1^2-x_2^2+2x_1x_2\e_{12}).
\end{aligned} 
\end{equation}
\end{proof}
We want to determine what angle is encoded in the expression (\ref{assertion}) and look at some examples.
\begin{ex}\label{b:saddle+vortex}
For the sum of a saddle and a vortex $\v(\x)=\a(\x)+\vec d(\x)%=%(x_1+x_2)\e_1+(-x_2-x_1)\e_2=
%-\e_2\x\e_2+\e_{12}\x
$ the totally rotated copy suffices
\begin{equation}
\begin{aligned}
%\u(\x)=&
\operatorname{R} _{\alpha}(\v(\operatorname{R} _{-\alpha}(\x)))
=&\operatorname{R} _{\alpha}(\a(\operatorname{R} _{-\alpha}(\x)))+\operatorname{R} _{\alpha}(\vec d(\operatorname{R} _{-\alpha}(\x)))
\\=&e^{-2\alpha \e_{12}}\a(\x)+\vec d(\x)
\end{aligned} 
\end{equation}
%and for $\alpha\in(-\frac\pi2,\frac\pi2)$ we get from Lemmata \ref{l:v1} and \ref{l:v2} that 
so their product takes the shape
% \begin{equation}
% \begin{aligned}
% \operatorname{R} _{\alpha}((\a+\vec d)(\operatorname{R} _{-\alpha}(\x)))(\a(\x)+\vec d(\x))=&2(x_1+x_2)(\cos(\alpha)(x_1+x_2)+\sin(\alpha)(-x_1+x_2))e^{-\alpha \e_{12}}
% \end{aligned} 
% \end{equation}
% because 
\begin{equation}
\begin{aligned}
& R_{\alpha}(\v(R_{-\alpha}(\x)))\v(\x)
\\ \stackrel{(5.5)}{=}& e^{-2\alpha\e_{12}}\x^2- e^{-2\alpha\e_{12}}(-2x_1x_2+(x_1^2-x_2^2)\e_{12})
\\ &- (-2x_1x_2-(x_1^2-x_2^2)\e_{12}) 
 + e^{\alpha\e_{12}}\x^2e^{-\alpha\e_{12}}
\\=& \big((e^{-\alpha\e_{12}}+e^{\alpha\e_{12}})     (\x^2+2x_1x_2)
\\& -(e^{-\alpha\e_{12}}-e^{\alpha\e_{12}})
     (x_1^2-x_2^2)\e_{12}\big) e^{-\alpha\e_{12}} 
 \\=& \big(2 \cos (\alpha)(x_1+x_2)^2+2\sin (\alpha)(x_1^2-x_2^2)\big)e^{-\alpha\e_{12}}.
% %&\operatorname{R} _{\alpha}(\vec{a}(\operatorname{R} _{-\alpha}\vec{x})+\vec{d}(\operatorname{R} _{-\alpha}\vec{x}))(\vec{a}%(\vec{x})+\vec{d}(\vec{x}))
% %\\=& e^{-2\alpha\e_{12}}\vec{a}(\vec{x})^2 
% %     + e^{-2\alpha\e_{12}}\vec{a}(\vec{x})\vec{d}(\vec{x})
% %     + \vec{d}(\vec{x}) \vec{a}(\vec{x})
% %+\vec{d}(\vec{x})^2
% &\operatorname{R} _{\alpha}(\v(\operatorname{R} _{-\alpha}(\x)))\v(\x)
% \\\overset{(\ref{detail})}=& e^{-2\alpha\e_{12}}\vec{x}^2
%      + e^{-2\alpha\e_{12}}(-2x_1x_2 + (x_1^2-x_2^2) \e_{12})
% \\&     + (-2x_1x_2 - (x_1^2-x_2^2) \e_{12})
% %  + \vec{x}^2 
% % \\=& (e^{-\alpha\e_{12}}\vec{x}^2
% %      + e^{-\alpha\e_{12}}(-2x_1x_2 + (x_1^2-x_2^2) \e_{12})
% % \\&     + e^{\alpha\e_{12}}(-2x_1x_2 - (x_1^2-x_2^2) \e_{12})
%  + e^{\alpha\e_{12}}\vec{x}^2 )e^{-\alpha \e_{12}}
% \\=& (e^{-\alpha\e_{12}}+ e^{\alpha\e_{12}})(\vec{x}^2-2x_1x_2)
% \\&      + (e^{-\alpha\e_{12}}- e^{\alpha\e_{12}}) (x_1^2-x_2^2) \e_{12}
%   )e^{-\alpha\e_{12}}
% \\=& (2\cos(\alpha)(x_1^2+x_2^2)
%     + 2\sin(\alpha) (x_1^2-x_2^2) 
%   )e^{-\alpha\e_{12}}.
\end{aligned} 
\end{equation}
% \begin{equation}
% \begin{aligned}
% %&(e^{-2\alpha \e_{12}}\a(\x)+\b(\x))\c(\x)
% %&\operatorname{R} _{\alpha}(\c(\operatorname{R} _{-\alpha}(\x)))\c(\x)%(e^{-2\alpha \e_{12}}(x_1\e_1-x_2\e_2)+(x_2\e_1-x_1\e_2))((x_1\e_1-x_2\e_2)+(x_2\e_1-x_1\e_2))
% =&e^{-2\alpha \e_{12}}\a(\x)^2+e^{-2\alpha \e_{12}}\a(\x)\vec d(\x)+\vec d(\x)\a(\x)+\vec d(\x)^2
% \\=&e^{-2\alpha \e_{12}}(x_1\e_1-x_2\e_2)^2+e^{-2\alpha \e_{12}}(x_1\e_1-x_2\e_2)(x_2\e_1-x_1\e_2)
% \\&+(x_2\e_1-x_1\e_2)(x_1\e_1-x_2\e_2)+(x_2\e_1-x_1\e_2)^2
% \\=&e^{-2\alpha \e_{12}}(x_1^2+x_2^2)+e^{-2\alpha \e_{12}}(2x_1x_2+(-x_1^2+x_2^2)\e_{12})+2x_1x_2+(x_1^2-x_2^2)\e_{12}+(x_1^2+x_2^2)
% \\=&e^{-2\alpha \e_{12}}(x_1^2+2x_1x_2+x_2^2)+e^{-2\alpha \e_{12}}(-x_1^2+x_2^2)\e_{12}+(x_1^2-x_2^2)\e_{12}+(x_1^2+2x_1x_2+x_2^2)
% \\=&e^{-2\alpha \e_{12}}(x_1+x_2)^2+(e^{-2\alpha \e_{12}}-1)(-x_1+x_2)(x_1+x_2)\e_{12}
% \\=&(x_1+x_2)((x_1+x_2)(e^{-2\alpha \e_{12}}+1)+(e^{-2\alpha \e_{12}}-1)(-x_1+x_2)\e_{12})
% \\=&(x_1+x_2)((x_1+x_2)(\cos(2\alpha)-\sin(2\alpha)\e_{12}+1)+(\cos(2\alpha)-\sin(2\alpha)\e_{12}-1)(-x_1+x_2)\e_{12})
% \\=&(x_1+x_2)((x_1+x_2)(2\cos^2(\alpha)-2\sin(\alpha)\cos(\alpha)\e_{12})
% \\&+(-x_1+x_2)\e_{12}(-2\sin^2(\alpha)-2\sin(\alpha)\cos(\alpha)\e_{12}))
% \\=&(x_1+x_2)(\cos(\alpha)(x_1+x_2)(2\cos(\alpha)-2\sin(\alpha)\e_{12})
% \\&+\sin(\alpha)(-x_1+x_2)(-2\sin(\alpha)\e_{12}+2\cos(\alpha)))
% \\=&2(x_1+x_2)(\cos(\alpha)-\sin(\alpha)\e_{12})(\cos(\alpha)(x_1+x_2)+\sin(\alpha)(-x_1+x_2))
% \\=&2(x_1+x_2)(\cos(\alpha)(x_1+x_2)+\sin(\alpha)(-x_1+x_2))e^{-\alpha \e_{12}}.
% \end{aligned} 
% \end{equation}
Its inverse reveals the correct misalignment by $\alpha$.
\end{ex}
All the examples we had might lead to the assumption, that the argument $\varphi$ of the polar form of the geometric product always takes a value in $[0,-2\alpha]$. But the following counterexample shows that this is not true.
\begin{ex}
 The product of the linear vector field $\v(\x)=\a(\x)+2\c(\x)=3x_1\e_1+x_2\e_2$, which is a superposition of a saddle and a source, and its copy rotated by $\alpha=\frac\pi4$ has the value 
\begin{equation}
\begin{aligned}
%&\operatorname{R} _{\alpha}((\a+2\c)(\operatorname{R} _{-\alpha}(\x)))(\a(\x)+2\c(\x))
&\operatorname{R} _{\alpha}(\v(\operatorname{R} _{-\alpha}(\x)))\v(\x)
\\=&((2x_1+x_2)\e_1+(x_1+2x_2)\e_2)(3x_1\e_1+x_2\e_2)
\\=&(2(3x_1^2+2x_1x_2+x_2^2))-(3x_1^2+4x_1x_2-x_2^2)\e_{12}
\end{aligned} 
\end{equation}
If we evaluate it at the position $\x=-\e_1+\e_2$  we get
\begin{equation}
\begin{aligned}
%\operatorname{R} _{\alpha}((\a+2\c)(\operatorname{R} _{-\alpha}(-e_1+e_2)))(\a(-e_1+e_2)+2\c(-e_1+e_2))
&\operatorname{R} _{\alpha}(\v(\operatorname{R} _{-\alpha}(-\e_1+\e_2)))\v(-\e_1+\e_2)
%&\operatorname{R} _{\alpha}((\a+2\c)(\operatorname{R} _{-\alpha}(\x)))(\a(\x)+2\c(\x))
=4+2\e_{12},
\end{aligned} 
\end{equation}
the argument of which is positive. %That means we would move even further away.
% As a side note: rotating iteratively would converge to a rotational misalignment of $\pi$. This would actually be fine, because the vector field is symmetric.
\end{ex}
In the coming sections we will show that in contrast to the product the assumption will hold for the geometric correlation over a symmetric area.
%----------------------------------------------------------------------------------------------------------------------------
\section{Influence of Superposition on the Correlation}
%----------------------------------------------------------------------------------------------------------------------------
Heuristic shows, that a case like the last example appears relatively sparsely. Using the average over a larger area could erase such appearances. We will show that the argument $\varphi$ is always in $[0,-2\alpha]$ if we take the integral of the product over an area $A$  symmetric with respect to both coordinate axes, like a square or a circle. This integral is equivalent to the correlation at the origin, if we assume the vector fields to vanish outside this area.
\begin{thm}\label{t:cor}
 Let the two-dimensional vector field $\v(\x)$ be linear within and zero outside of an area $A$ symmetric with respect to both coordinate axes. The correlation at the origin with its totally rotated copy $\u(\x)=\operatorname{R} _{\alpha}(\v(\operatorname{R} _{-\alpha}(\x)))$ satisfies
\begin{equation}
\begin{aligned}\label{assertion:cor}
(\u\star\v)(0)
&=e^{-2\alpha \e_{12}}||\v_1(\x)||_{L^2(A)}^2
+||\v_2(\x)||_{L^2(A)}^2
\end{aligned} 
\end{equation}
with $\v_1(\x)=(a-b\e_{12})(-\e_2\x\e_2),\v_2(\x)=(c+d\e_{12})\x$ from Lemma \ref{l:v1,v2}.
\end{thm}
\begin{proof}
We already know from Lemma \ref{l:product} that the product of the vector field and its rotated copy takes the form
\begin{equation}
\begin{aligned}
\u(\x)\v(\x)=&e^{-2\alpha \e_{12}}\v_1(\x)^2+e^{-2\alpha \e_{12}}\v_1(\x)\v_2(\x)
\\&+\v_2(\x)\v_1(\x)+\v_2(\x)^2.
\end{aligned} 
\end{equation}
% so integration leads to
% \begin{equation}
% \begin{aligned}
% %&\int_{[-l,l]^2}\operatorname{R} _{\alpha}((\v_1+\v_2)(\operatorname{R} _{-\alpha}(\x)))(\v_1(\x)+\v_2(\x))\d^2 x
% \operatorname{R} _{\alpha}(\v(\operatorname{R} _{-\alpha}(\x)))\v(\x) 
% =&\int_{A}e^{-2\alpha \e_{12}}\v_1(\x)^2+\v_2(\x)\v_1(\x)
% \\&+e^{-2\alpha \e_{12}}\v_1(\x)\v_2(\x)+\v_2(\x)^2 \d^2\x.
% \end{aligned} 
% \end{equation}
Taking into account (\ref{detail}) and the fact, that the integral over the symmetric domain $A$ over $x_1^2-x_2^2$ is zero as well as the integral over $x_1x_2$, we get
\begin{equation}
\begin{aligned}\label{istnull}
% %\int_{[-l,l]^2}\v(\x)^2\d^2 x=&\int_{[-l,l]^2}(a^2+b^2)(x_1^2+x_2^2)\d^2x
% %\\=&\frac{8r^4}3(a^2+b^2),\\
% \int_{[-l,l]^2}\v_1(\x)\v_2(\x)\d^2 x=&\int_{[-l,l]^2}(ac+bd)(x_1^2-x_2^2)+2(bc-ad)x_1x_2
% \\&+((ad-bc)(x_1^2-x_2^2)+2(ac+bd)x_1x_2)\e_{12}\d^2x
% \\=&0,\\
% %\int_{[-l,l]^2}\w(\x)^2\d^2 x=&\int_{[-l,l]^2}(c^2+d^2)(x_1^2+x_2^2)\d^2x
% %\\=&\frac{8r^4}3(c^2+d^2)
\int_{A}\v_1(\x)\v_2(\x)\d^2 x=&\int_{[-l,l]^2}(a-b \e_{12})(c-d\e_{12}) 
\\&\cdot(x_1^2-x_2^2+2x_1x_2 \e_{12})\d^2 \x
\\=&0
\end{aligned} 
\end{equation}
and $\int_{A}\v_2(\x)\v_1(\x)\d^2=0$, too. That is why the integral over the product reduces to 
\begin{equation}
\begin{aligned}
\int_{A}\u(\x)\v(\x)\d^2 x
=&\int_{A}e^{-2\alpha \e_{12}}\v_1(\x)^2+\v_2(\x)\v_1(\x)
 \\&+e^{-2\alpha \e_{12}}\v_1(\x)\v_2(\x)+\v_2(\x)^2 \d^2\x
\\=&e^{-2\alpha \e_{12}}||\v_1(\x)||_{L^2(A)}^2
+||\v_2(\x)||_{L^2(A)}^2.
\end{aligned} 
\end{equation}
\end{proof}
\begin{rem}
 Please note that an integral over an unsymmetric area does in general not lead to a result without the mixed terms $\v_1(\x)\v_2(\x)$.
\end{rem}

%----------------------------------------------------------------------------------------------------------------------------
\section{Detection of the Angle}
%----------------------------------------------------------------------------------------------------------------------------
Now we want to use Theorem \ref{t:cor} to evaluate the angle $\alpha$ by which our pattern and our vector field differ.
%----------------------------------------------------------------------------------------------------------------------------
%\subsection{Knowing the Pattern Analytically}
%----------------------------------------------------------------------------------------------------------------------------
First assume we have the analytical description of the pattern. 
\begin{cor}
If we know the shape of $\v_1(\x)$ and $\v_2(\x)$ from Lemma \ref{l:v1,v2}, we can determine the angle $\alpha$ of the rotational misalignment of the reference pattern and the vector field from
\begin{equation}
\begin{aligned}
\alpha=&-\frac12\arg((\u\star\v_1)(0)).
\end{aligned} 
\end{equation}
\end{cor}
Please note, that knowledge about the other vector field, the counterpart to the pattern, is not necessary and that we usually have analytic information about the pattern, because we generally know what we are looking for.
%----------------------------------------------------------------------------------------------------------------------------
%\subsection{Without Analytical Information}
%----------------------------------------------------------------------------------------------------------------------------
\par
If we do not have the analytic description of the pattern, we can still use the correlation, because the argument is a more or less good approximation to the true rotational difference $\alpha$.
% Iterative application of the rotor could converge towards the correct misalignment.
%
% \begin{algorithm}
% \caption{Intuitive detection of total misalignment of vector fields}
% \label{alg1}
% \begin{algorithmic}[1]
% \REQUIRE vector field: $\v(\x)$, rotated pattern: $\u(\x)$, desired accuracy: $\varepsilon>0$,
%  \STATE $\varphi=\pi,\alpha=0$,
% \WHILE{$\varphi>\varepsilon$}
%  \STATE $Cor=(\u(\x)\star \v(\x))(0)$,
% \STATE $\varphi=\arg(Cor)$,
%  \STATE $\alpha=\alpha+\varphi$,
%  \STATE $\u(\x)=e^{-\varphi \e_{12}}\u(e^{\varphi \e_{12}}\x)$,
% \ENDWHILE
% \ENSURE misalignment: $\alpha$, corrected pattern: $\u(\x)$.
% \end{algorithmic}
% \end{algorithm}
%
%
%
\begin{lem} \label{l:phi}
 Let the two-dimensional vector field $\v(\x)$ be linear within and zero outside of an area $A$ symmetric with respect to both coordinate axes. The angle $\varphi$ which is the argument of the correlation at the origin with its totally rotated copy $\u(\x)=\operatorname{R} _{\alpha}(\v(\operatorname{R} _{-\alpha}(\x)))$ satisfies
%  Let $\alpha\in(-\pi/2,\pi/2)$. The angle $\varphi$ which is the argument of the polar form $re^{\varphi\e_{12}}=\int_{[-l,l]^2}\operatorname{R} _{\alpha}(\v(\operatorname{R} _{-\alpha}(\x)))\v(\x)\d^2 x$ of the symmetric integral over the product of any two-dimensional linear vector field and its by $\alpha$ rotated copy satisfies
\begin{equation}
\begin{aligned}
0\geq\varphi\geq-2\alpha,&\text{ for }\alpha\geq0,\\
0\leq\varphi\leq-2\alpha,&\text{ else.}
\end{aligned} 
\end{equation}
\end{lem}
The proof of Lemma \ref{l:phi} is very technical. Figure \ref{f:3} provides a more fundamental insight of its assertion by exploiting the homomorphism of the rotors in $\clifford{2,0}$ and the complex numbers.
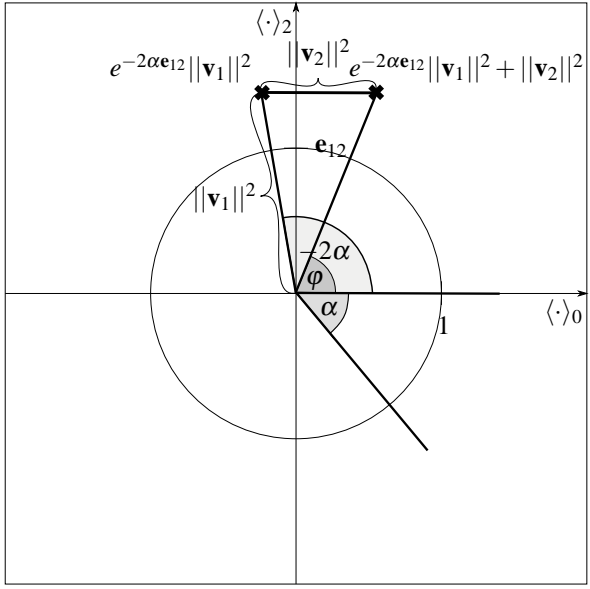
\begin{figure}[ht]\label{f:3}
%\begin{minipage}{0.25\textwidth}
%\centering
\psset{unit=0.55pt}
  %LaTeX with PSTricks extensions
%%Creator: inkscape 0.47
%%Please note this file requires PSTricks extensions
%\psset{xunit=.5pt,yunit=.5pt,runit=.5pt}
\begin{pspicture}(402,402)
{
\newrgbcolor{curcolor}{0 0 0}
\pscustom[linewidth=0.40126583,linecolor=curcolor]
{
\newpath
\moveto(0.42425418,400.33428955)
\lineto(400.42425418,400.33428955)
\lineto(400.42425418,0.33447266)
\lineto(0.42425418,0.33447266)
\lineto(0.42425418,400.33428955)
\closepath
}
}
{
\newrgbcolor{curcolor}{0 0 0}
\pscustom[linewidth=0.40000001,linecolor=curcolor]
{
\newpath
\moveto(0.42425346,200.33424262)
\lineto(400.42425,200.33424262)
}
}
{
\newrgbcolor{curcolor}{0 0 0}
\pscustom[linestyle=none,fillstyle=solid,fillcolor=curcolor]
{
\newpath
\moveto(396.42424994,200.33424262)
\lineto(394.82424992,198.73424259)
\lineto(400.42425,200.33424262)
\lineto(394.82424992,201.93424264)
\lineto(396.42424994,200.33424262)
\closepath
}
}
{
\newrgbcolor{curcolor}{0 0 0}
\pscustom[linewidth=0.40000001,linecolor=curcolor]
{
\newpath
\moveto(396.42424994,200.33424262)
\lineto(394.82424992,198.73424259)
\lineto(400.42425,200.33424262)
\lineto(394.82424992,201.93424264)
\lineto(396.42424994,200.33424262)
\closepath
}
}
{
\newrgbcolor{curcolor}{0 0 0}
\pscustom[linewidth=0.31950482,linecolor=curcolor]
{
\newpath
\moveto(300.42424904,200.33429313)
\curveto(300.42424904,145.10581793)(255.65272495,100.33429273)(200.42425111,100.33429273)
\curveto(145.19577728,100.33429273)(100.42425319,145.10581793)(100.42425319,200.33429313)
\curveto(100.42425319,255.56276833)(145.19577728,300.33429352)(200.42425111,300.33429352)
\curveto(255.65272495,300.33429352)(300.42424904,255.56276833)(300.42424904,200.33429313)
\closepath
}
}
{
\newrgbcolor{curcolor}{0.87058824 0.87058824 0.87058824}
\pscustom[linestyle=none,fillstyle=solid,fillcolor=curcolor]
{
\newpath
\moveto(223.74339,172.78584262)
\curveto(234.14187,182.35704262)(235.95019,187.60264262)(236.56241,200.41904262)
\lineto(200.73361,200.51104262)
\curveto(200.73361,200.51104262)(226.74933,170.20224262)(226.57183,169.33864262)
\lineto(226.57183,169.33964262)
\lineto(223.74339,172.78604262)
\closepath
}
}
{
\newrgbcolor{curcolor}{0 0 0}
\pscustom[linewidth=0.80000001,linecolor=curcolor]
{
\newpath
\moveto(223.74339,172.78584262)
\curveto(234.14187,182.35704262)(235.95019,187.60264262)(236.56241,200.41904262)
\lineto(200.73361,200.51104262)
\curveto(200.73361,200.51104262)(226.74933,170.20224262)(226.57183,169.33864262)
\lineto(226.57183,169.33964262)
\lineto(223.74339,172.78604262)
\closepath
}
}
{
\newrgbcolor{curcolor}{0 0 0}
\pscustom[linewidth=2,linecolor=curcolor]
{
\newpath
\moveto(176.9789,338.69294262)
\lineto(257.64287,338.40014262)
}
}
{
\newrgbcolor{curcolor}{0 0 0}
\pscustom[linewidth=0.40445113,linecolor=curcolor]
{
\newpath
\moveto(200.52913,200.66544262)
\curveto(175.70967,195.47944262)(201.57185,273.34698262)(176.01988,267.26738262)
\curveto(201.47703,272.95622262)(151.27909,333.49698262)(176.27272,338.44154262)
}
}
{
\newrgbcolor{curcolor}{0 0 0}
\pscustom[linewidth=0.49681437,linecolor=curcolor]
{
\newpath
\moveto(255.50993,341.69762262)
\curveto(255.82777,356.39526262)(214.96041,333.39486262)(215.70191,348.60482262)
\curveto(215.16913,333.49034262)(176.07816,355.68350262)(175.91349,340.91218262)
}
}
{
\newrgbcolor{curcolor}{0 0 0}
\pscustom[linewidth=4,linecolor=curcolor]
{
\newpath
\moveto(173.01354,334.62002262)
\lineto(181.04925,342.47714262)
}
}
{
\newrgbcolor{curcolor}{0 0 0}
\pscustom[linewidth=4,linecolor=curcolor]
{
\newpath
\moveto(173.01354,342.47714262)
\lineto(181.13854,334.62002262)
}
}
{
\newrgbcolor{curcolor}{0 0 0}
\pscustom[linewidth=4,linecolor=curcolor]
{
\newpath
\moveto(251.51193,334.87322262)
\lineto(259.54765,342.73034262)
}
}
{
\newrgbcolor{curcolor}{0 0 0}
\pscustom[linewidth=4,linecolor=curcolor]
{
\newpath
\moveto(251.51193,342.73034262)
\lineto(259.63693,334.87322262)
}
}
{
\newrgbcolor{curcolor}{0.94117647 0.94117647 0.93725491}
\pscustom[linestyle=none,fillstyle=solid,fillcolor=curcolor]
{
\newpath
\moveto(200.50335,200.76784262)
\curveto(201.98707,200.67584262)(253.13949,200.55984262)(253.05971,200.50784262)
\curveto(249.41249,247.64982262)(209.16017,256.34174262)(191.56531,252.39886262)
\curveto(194.20414,234.86378262)(197.30825,218.22606262)(200.50335,200.76784262)
\closepath
}
}
{
\newrgbcolor{curcolor}{0 0 0}
\pscustom[linewidth=1.00711989,linecolor=curcolor]
{
\newpath
\moveto(200.50335,200.76784262)
\curveto(201.98707,200.67584262)(253.13949,200.55984262)(253.05971,200.50784262)
\curveto(249.41249,247.64982262)(209.16017,256.34174262)(191.56531,252.39886262)
\curveto(194.20414,234.86378262)(197.30825,218.22606262)(200.50335,200.76784262)
\closepath
}
}
{
\newrgbcolor{curcolor}{0 0 0}
\pscustom[linewidth=0.40000001,linecolor=curcolor]
{
\newpath
\moveto(200.42425,0.33418262)
\lineto(200.42425,400.33430262)
}
}
{
\newrgbcolor{curcolor}{0 0 0}
\pscustom[linestyle=none,fillstyle=solid,fillcolor=curcolor]
{
\newpath
\moveto(200.42425,396.33430256)
\lineto(202.02425002,394.73430253)
\lineto(200.42425,400.33430262)
\lineto(198.82424998,394.73430253)
\lineto(200.42425,396.33430256)
\closepath
}
}
{
\newrgbcolor{curcolor}{0 0 0}
\pscustom[linewidth=0.40000001,linecolor=curcolor]
{
\newpath
\moveto(200.42425,396.33430256)
\lineto(202.02425002,394.73430253)
\lineto(200.42425,400.33430262)
\lineto(198.82424998,394.73430253)
\lineto(200.42425,396.33430256)
\closepath
}
}
{
\newrgbcolor{curcolor}{0 0 0}
\pscustom[linewidth=1.60000002,linecolor=curcolor]
{
\newpath
\moveto(200.21129,201.11104262)
\lineto(176.41396,338.97918262)
}
}
{
\newrgbcolor{curcolor}{0.78431374 0.78431374 0.78431374}
\pscustom[linestyle=none,fillstyle=solid,fillcolor=curcolor]
{
\newpath
\moveto(210.53375,225.76186262)
\lineto(200.38179,201.21024262)
\curveto(200.38179,201.21024262)(227.91513,200.58544262)(227.54841,201.07024262)
\curveto(227.66641,216.48298262)(215.62191,223.94266262)(210.53375,225.76186262)
\closepath
}
}
{
\newrgbcolor{curcolor}{0 0 0}
\pscustom[linewidth=0.62353176,linecolor=curcolor]
{
\newpath
\moveto(210.53375,225.76186262)
\lineto(200.38179,201.21024262)
\curveto(200.38179,201.21024262)(227.91513,200.58544262)(227.54841,201.07024262)
\curveto(227.66641,216.48298262)(215.62191,223.94266262)(210.53375,225.76186262)
\closepath
}
}
{
\newrgbcolor{curcolor}{0 0 0}
\pscustom[linewidth=1.60000002,linecolor=curcolor]
{
\newpath
\moveto(290.91011,92.29898262)
\lineto(200.29555,201.22624262)
\lineto(256.13855,338.90570262)
}
}
{
\newrgbcolor{curcolor}{0 0 0}
\pscustom[linewidth=1.60000002,linecolor=curcolor]
{
\newpath
\moveto(200.58607,200.69784262)
\lineto(340.49201,200.19304262)
}
}
{
\newrgbcolor{curcolor}{0 0 0}
\pscustom[linewidth=0.40000001,linecolor=curcolor]
{
\newpath
\moveto(192.18161,300.33430262)
\lineto(208.61511,300.28230262)
}
}
{
\newrgbcolor{curcolor}{0 0 0}
\pscustom[linewidth=0.40000001,linecolor=curcolor]
{
\newpath
\moveto(300.42425,192.45904262)
\lineto(300.42425,208.70942262)
}
\rput[r](173.625,266.125){$||\v_1||^2$}
\rput[b](215.89044416,355.75000499){$||\v_2||^2$}
\rput(223.25,189.5625){$\alpha$}
\rput(220,229.5){$-2\alpha$}
\rput(213.5,209.0625){$\varphi$}
\rput[rb](170.25,345.3125){$e^{-2\alpha\e_{12}}||\v_1||^2$}
\rput[lb](236.875,345.375){$e^{-2\alpha\e_{12}}||\v_1||^2+||\v_2||^2$}
\rput[r](198.625,387.125){$\langle\cdot\rangle_2$}
\rput[t](384.625,196.125){$\langle\cdot\rangle_0$}
\rput[l](213.5,300.25){$\e_{12}$}
\rput[t](302,185){$1$}
}
\end{pspicture}
%\end{minipage}
\caption{Lemma \ref{l:phi} visualized like the complex plane. Vertical axis: bivector part, horizontal axis: scalarpart.}
\end{figure}
\begin{proof}
The argument satisfies
\begin{equation}
\begin{aligned}
\varphi=&\arg(\int_{[-l,l]^2}\operatorname{R} _{\alpha}(\v(\operatorname{R} _{-\alpha}(\x)))\v(\x)\d^2 x)
\\=&\arg(e^{-2\alpha \e_{12}}||\v_1(\x)||^2+||\v_2(\x)||^2)
\\=&\operatorname{atan2}(-\sin(2\alpha)||\v_1(\x)||^2,
\\&\cos(2\alpha)||\v_1(\x)||^2+||\v_2(\x)||^2)
\end{aligned} 
\end{equation}
For $||\v_1(\x)||^2=0$ and $||\v_2(\x)||^2=0$ the statement is trivially true, because then $\varphi=0$ or $\varphi=-2\alpha$. So let $||\v_1(\x)||^2,||\v_2(\x)||^2>0$. Now we have to make a case differentiation. 
\begin{enumerate}
 \item The assumptions $\cos(2\alpha)||\v_1(\x)||^2+||\v_2(\x)||^2>0$ and $ -\sin(2\alpha)||\v_1(\x)||^2>0$ lead to
\begin{equation}
\begin{aligned}
\varphi=&\arctan\frac{-\sin(2\alpha)||\v_1(\x)||^2}{\cos(2\alpha)||\v_1(\x)||^2+||\v_2(\x)||^2}
\end{aligned} 
\end{equation}
so $\varphi$ is positive.
If we leave out $||\v_2(\x)||^2$ the denominator gets smaller. If the denominator $\cos(2\alpha)||\v_1(\x)||^2>0$ remeins positive the positive fraction gets larger and we have
\begin{equation}
\begin{aligned}
\varphi\leq&\arctan\frac{-\sin(2\alpha)||\v_1(\x)||^2}{\cos(2\alpha)||\v_1(\x)||^2}
=-2\alpha,
\end{aligned} 
\end{equation}
with positive $-2\alpha$ and therefore negative $\alpha$. If the denominator $\cos(2\alpha)||\v_1(\x)||^2\leq0$ becomes negative we have $-2\alpha\in[\frac{\pi}2,\pi]$, because of $ -\sin(2\alpha)||\v_1(\x)||^2>0$, so $\varphi\in(-\frac{\pi}2,\frac{\pi}2)\leq-2\alpha$.
 \item The assumptions $\cos(2\alpha)||\v_1(\x)||^2+||\v_2(\x)||^2>0$ and $-\sin(2\alpha)||\v_1(\x)||^2<0$ lead to 
\begin{equation}
\begin{aligned}
\varphi=&\arctan\frac{-\sin(2\alpha)||\v_1(\x)||^2}{\cos(2\alpha)||\v_1(\x)||^2+||\v_2(\x)||^2}
\end{aligned} 
\end{equation}
so $\varphi$ is negative.
If we leave out $||\v_2(\x)||^2$ the denominator gets smaller. If the denominator $\cos(2\alpha)||\v_1(\x)||^2>0$ remeins positive the negative fraction gets smaller and we have
\begin{equation}
\begin{aligned}
\varphi\geq&\arctan\frac{-\sin(2\alpha)||\v_1(\x)||^2}{\cos(2\alpha)||\v_1(\x)||^2}
=-2\alpha,
\end{aligned} 
\end{equation}
with negative $-2\alpha$ and therefore positive $\alpha$. If the denominator $\cos(2\alpha)||\v_1(\x)||^2\leq0$ becomes negative we have $-2\alpha\in[-\pi,-\frac{\pi}2]$, because of $ -\sin(2\alpha)||\v_1(\x)||^2<0$, so $\varphi\in(-\frac{\pi}2,\frac{\pi}2)\geq-2\alpha$.
 \item The assumptions $\cos(2\alpha)||\v_1(\x)||^2+||\v_2(\x)||^2<0$ and $-\sin(2\alpha)||\v_1(\x)||^2>0$ lead to 
\begin{equation}
\begin{aligned}
\varphi=&\arctan\frac{-\sin(2\alpha)||\v_1(\x)||^2}{\cos(2\alpha)||\v_1(\x)||^2+||\v_2(\x)||^2}+\pi
%\\=&\arctan\frac{\sin(2\alpha)||\v_1(\x)||^2}{-\cos(2\alpha)||\v_1(\x)||^2-||\v_2(\x)||^2}+\pi
\end{aligned} 
\end{equation}
so $\varphi$ is positive.
If we leave out $||\v_2(\x)||^2$ the magnitude of the denominator gets larger so the magnitude of the fraction gets smaller. Since the fraction is negative and the arctangent is monotonic increasing a lower magnitude increases the whole right side and we have
\begin{equation}
\begin{aligned}
\varphi\leq&\arctan\frac{-\sin(2\alpha)}{\cos(2\alpha)}+\pi=-2\alpha.
\end{aligned} 
\end{equation}
Because the numerator is positive and the denominator is negative this equals $-2\alpha$, which is positive and therefore $\alpha$ is negative.
 \item The assumptions $\cos(2\alpha)||\v_1(\x)||^2+||\v_2(\x)||^2<0$ and $-\sin(2\alpha)||\v_1(\x)||^2<0$ lead to 
\begin{equation}
\begin{aligned}
\varphi=&\arctan\frac{-\sin(2\alpha)||\v_1(\x)||^2}{\cos(2\alpha)||\v_1(\x)||^2+||\v_2(\x)||^2}-\pi
\end{aligned} 
\end{equation}
so $\varphi$ is negative.
If we leave out $||\v_2(\x)||^2$ the magnitude of the denominator gets larger so the magnitude of the fraction decreases. It is positive so the fraction gets smaller, so does the arctangent and the whole right side and we have
\begin{equation}
\begin{aligned}
\varphi\geq&\arctan\frac{-\sin(2\alpha)}{\cos(2\alpha)}
-\pi=-2\alpha.
\end{aligned} 
\end{equation}
Because the numerator and the denominator are negative this equals $-2\alpha$, which is negative and therefore $\alpha$ is positive.
\end{enumerate}
Since we covered all possible configurations, we see that $\alpha$ and $\varphi$ always have different signs. The right estimation for positive $\alpha$ is a result of the even cases and for negative $\alpha$ of the odd ones.
\end{proof}
\begin{thm}\label{t:conv}
 Let the two-dimensional vector field $\v(\x)$ be linear within and zero outside of an area $A$ symmetric with respect to both coordinate axes and $\varphi:(-\frac\pi2,\frac\pi2)\to(-\frac\pi2,\frac\pi2)$ be the function defined by the rule 
\begin{equation}
\begin{aligned}
\varphi(\alpha)=\arg((\operatorname{R} _{\alpha}(\v(\operatorname{R} _{-\alpha}))\star \v)(0).
\end{aligned} 
\end{equation}
Then the series $\alpha_0=\alpha,\alpha_{n+1}=\alpha_{n}+\varphi(\alpha_{n})$ converges to zero for all $\alpha\in(-\frac\pi2,\frac\pi2)$, if $||\v_1(\x)||^2\neq0\neq||\v_2(\x)||^2$.
\end{thm}
\begin{proof}
Lemma \ref{l:phi} shows that the series $\alpha_0=\alpha,\alpha_{n+1}=\alpha_{n}+\varphi(\alpha_{n})$ decreases with respect to its magnitude, because for $\alpha_n\in(-\frac\pi2,0)$ we have $0\leq\varphi(\alpha_n)\leq-2\alpha_n$ and therefore
\begin{equation}
\begin{aligned}
\alpha_{n}
=\alpha_{n}+0
&\leq\alpha_{n}+\varphi(\alpha_{n})
=\alpha_{n+1},
\\\alpha_{n+1}
=\alpha_{n}+\varphi(\alpha_{n})
&\leq\alpha_{n}-2\alpha_{n}
=-\alpha_{n}
\end{aligned} 
\end{equation}
 and for $\alpha_n\in(0,\frac\pi2)$ we have $0\geq\varphi(\alpha_n)\geq-2\alpha_n$ and therefore
\begin{equation}
\begin{aligned}
\alpha_{n}
=\alpha_{n}+0
&\geq\alpha_{n}+\varphi(\alpha_{n})
=\alpha_{n+1},
\\\alpha_{n+1}
=\alpha_{n}+\varphi(\alpha_{n})
&\geq\alpha_{n}-2\alpha_{n}
=-\alpha_{n}.
\end{aligned} 
\end{equation}
Since the series of magnitudes is monotonically decreasing and bounded from below by zero it is convergent. 
\par
Let the limit of the sequence of magnitudes be $a=\lim_{n\to\infty}|\alpha_n|$ then using the definition of the series and applying the limit leads to
\begin{equation}
\begin{aligned}
\lim_{n\to\infty}(|\alpha_{n+1}|)=\lim_{n\to\infty}(|\alpha_{n}+\varphi(\alpha_{n})|).
\end{aligned} 
\end{equation}
The modulus function and $\varphi(\alpha_n)$ are continuous in $\alpha_n\in(-\frac\pi2,\frac\pi2)$. That allows us to swap the limit and the functions and write
\begin{equation}
\begin{aligned}
a=&|\lim_{n\to\infty}(\alpha_{n})+\lim_{n\to\infty}(\varphi(\alpha_{n}))|
\\=&|a+\varphi(a)|.
\end{aligned} 
\end{equation}
We apply a case differentiation to the previous equation.
\begin{enumerate}
 \item For $a+\varphi(a)\geq0$ it is equivalent to
\begin{equation}
\begin{aligned}
a=&a+\varphi(a)\Leftrightarrow\varphi(a)=0.
\end{aligned} 
\end{equation}
Since \begin{equation}
\begin{aligned}
\varphi(\alpha)=&\operatorname{atan2}(-\sin(2\alpha)||\v_1(\x)||^2,
\\&\cos(2\alpha)||\v_1(\x)||^2+||\v_2(\x)||^2)
\end{aligned} 
\end{equation}
the claim $\varphi(a)=0$ is true for $\cos(2a)||\v_1(\x)||^2+||\v_2(\x)||^2>0,-\sin(2a)||\v_1(\x)||^2=0$ which is fulfilled either for $||\v_1(\x)||^2=0$ and arbitrary $a$ or for $||\v_1(\x)||^2>0$ and $a=0$.
\item $a+\varphi(a)<0$ leads to
\begin{equation}
\begin{aligned}
a=&-a-\varphi(a)\Leftrightarrow\varphi(a)=-2a,
\end{aligned} 
\end{equation}
which is only fulfilled for $||\v_2(\x)||^2=0$ and arbitrary $a$.
\end{enumerate}
Combination of the two cases leads to the proposition $a=0$ if $||\v_1(\x)||^2\neq0\neq||\v_2(\x)||^2$. Since the sequence of the magnitudes converges to zero the sequence itself converges to zero as well.
\end{proof}
%
%----------------------------------------------------------------------------------------------------------------------------
\section{Algorithm and Experiments}
%----------------------------------------------------------------------------------------------------------------------------
The claim $||\v_1(\x)||^2=0$ means $\v_1(\x)=0$ almost everywhere. For a linear vector field this is equivalent to $\v_1(\x)=0$, analogously $||\v_2(\x)||^2=0\Leftrightarrow \v_2(\x)=0$. In the case $\v_1(\x)=0$ Lemma \ref{l:v1,v2} shows that $\forall \alpha\in(-\frac\pi2,\frac\pi2):\varphi(\alpha)=0$. An iterative algorithm would stop after one step and return the correct result, because these vector fields are rotational invariant anyway. 
\par
In the case $\v_2(\x)=0$ Lemma \ref{l:v1,v2} shows that $\forall \alpha\in(-\frac\pi2,\frac\pi2):\varphi(\alpha)=-2\alpha$. The algorithm would alternate between $-2\alpha$ and zero. That means if the algorithm takes the value zero in the $\alpha$ variable after its first iteration the underlying vector field must be a saddle $\v(\x)=\v_1(\x)$ and the correct misalignment is half the calculated $\varphi$. This exception is handled in Line 11 in Algorithm \ref{alg2}.
% Theorem \ref{t:conv} shows that in all other cases the iterative algorithm would converge to zero.
\par
\begin{algorithm}
\caption{Detection of total misalignment of vector fields}
\label{alg2}
\begin{algorithmic}[1]
\REQUIRE vector field: $\v(\x)$, rotated pattern: $\u(\x)$, desired accuracy: $\varepsilon>0$,
 \STATE $\varphi=\pi,\alpha=0,iter=0,exception=false$,
\WHILE{$\varphi>\varepsilon$}
  \STATE $iter++$,
  \STATE $Cor=(\u(\x)\star \v(\x))(0)$,
  \STATE $\varphi=\arg(Cor)$,
  \STATE $\alpha=\alpha+\varphi$,
  \IF {$iter=1$ and $\alpha=0$}
    \STATE $\alpha=\varphi=\pi/4$,
    \STATE $exception=true$,
  \ENDIF
  \IF {$iter=2$ and not $exception$ and $\alpha=0$}    
    \STATE $\alpha=\varphi=\varphi/2,$
  \ENDIF
  \IF {$iter=2$ and $exception$ and $\varphi=-\pi/2$} 
    \STATE $\alpha=\pi/2,\varphi=\pi/4$
  \ENDIF
  \STATE $\u(\x)=e^{-\varphi \e_{12}}\u(e^{\varphi \e_{12}}\x)$,
\ENDWHILE
\ENSURE misalignment: $\alpha$, corrected pattern: $\u(\x)$, iterations needed: $iter$.
\end{algorithmic}
\end{algorithm} 
\begin{table}[hbt!]
\begin{tabular}{|l|r|r|r|r|r|}
\hline
determined accuracy $eps$	&0.1	& 0.01	& 0.001	& 0.0001	& 0.00001	 \rule [-1.2mm]{0mm}{5mm}\\\hline
%relative number of fails in \%	&28.04	& 0.13	& 0.002	& 0	&0 	\rule [-1.2mm]{0mm}{5mm}\\
average error			&0.098	& 0.015	& 0.002	& 0.0002	& 0.00002	\rule [-1.2mm]{0mm}{5mm}\\\hline
maximal error			&1.927	& 0.903	& 0.312	& 0.089		& 0.028		\rule [-1.2mm]{0mm}{5mm}\\\hline
average number of iterations	&4.07	& 16.95	& 45.72	& 91.69		& 129.05		\rule [-1.2mm]{0mm}{5mm}\\\hline
\end{tabular}%\label{tab:g2}
\caption{
Results of Algorithm \ref{alg2} depending on the required accuracy.}
\end{table}
%
%Because of this, Remark \ref{r:pi/2}, and Theorem \ref{t:conv} we have now covered all linear vector fields and all angles $\alpha\in(-\pi,\pi]\setminus\{-\frac\pi2,\frac\pi2\}$. 
In the case of $\alpha=\pm \frac\pi2$ the correlation will be real valued, compare Theorem \ref{t:cor}. %Because of this and the fact that either the magnitude of the angle decreases in every step for the other cases or the algorithm stops, 
This case can only appear in the first step of the algorithm. It would return the angle zero like in the case where in deed no rotation is necessary. Therefore we need to include another exception handling. We suggest to apply a total rotation by $\frac\pi4$ to the pattern, if the first step returns $\alpha=0$, compare Line 7 in Algorithm \ref{alg2}. The disadvantage of this treatment is that it might disturb the alignment in the nice case, when vector field and pattern incidentally match at the beginning, but will guarantee the convergence.
%for any two-dimensional linear vector field and all $\alpha\in(-\pi,\pi]$. 
\par
The last exception to be treated appears when both $\alpha\in\{-\frac\pi2,0,\frac\pi2\}$ and $\v(\x)=\v_1(\x)$. In this case $\alpha$ gets the value $\frac\pi4$ from the first exception handling and will alternate between $\pm\frac\pi4$ for the rest of the algorithm. We fixed this problem in Line 14 in Algorithm \ref{alg2}.
\par
Together with Remark \ref{r:pi/2} this leads to the Corollary.
\begin{cor}
Algorithm \ref{alg2} returns the correct rotational misalignment for any two-dimensional linear vector field and its totally rotated copy by arbitrary angle.
\end{cor}
We practically tested Algorithm \ref{alg2} applying it to continuous, linear vector fields $\R^2\to\clifford{2,0}$, that vanish outside the unit square. The angle $\alpha\in(-\pi,\pi]$ and the four coefficients with magnitude not bigger than one describing the vector fields were determined randomly. The results for one million applications can be found in Table 1. The error was measured from the square root of the sum of the squared differences of the determined and the given coefficients. The experiments showed that 
%unsuccessful runs (i.e. the sum of the squared differences of determined and given coefficients exceeds 0.04) 
high numbers of necessary iterations occur when the magnitudes of $\v_1$ and $\v_2$ differ gravely. Table 1 also implies that the error decreases linearly with the required accuracy while the number of iterations increases sublinearly. But most importantly we could see Algorithm \ref{alg2} converges in all cases, just as the theory suggested. 
%----------------------------------------------------------------------------------------------------------------------------
\section{Conclusions and Outlook}
%----------------------------------------------------------------------------------------------------------------------------
The geometric cross correlation of two vector fields is scalar and bivector valued. We proved in Lemma \ref{l:phi} that for all linear two-dimensional vector fields this rotor has an argument with opposite sign and magnitude less or equal to twice the angle of the misalignment. Therefore application of the encoded rotation to the outer rotated copy of the vector field does not increase the misalignment to its original. In Theorem \ref{t:conv} we showed that iterative application completely erases the misalignment of the rotationally misaligned vector fields, if $||\v_1(\x)||\neq0\neq||\v_2(\x)||$. These exceptions could also be treated in Algorithm \ref{alg2}. We implemented it and experimentally confirmed the theoretic results.
\par
Currently we analyze the application of this approach to total rotations of three-dimensional vector fields. For praxis it will be interesting how the algorithm is able to treat vector fields, that are discrete, disturbed, or also dissimilar with respect to translation. The algorithm converges quite fast, if $||\v_1(\x)||$ and  $||\v_2(\x)||$ do not differ from each other too much. % If $||\v_1(\x)||$ is dominant, the angle will be overestimated and the algorithm oscillates, the calculated angles $\varphi$ are similar, but have opposite sign. If $||\v_2(\x)||$ is dominant, the angle will be underestimated and the detected $\varphi$ are very similar. 
In case one of them dominates it leads to oscillation or deceleration. We think about developing an adaptive algorithm, that estimates the ratio of $||\v_1(\x)||$ and $||\v_2(\x)||$ and weights the result respectively. Another way of improving the speed of the algorihtm is to use a fast Fourier transform and a geometric convolution theorem \cite{BSH12c}. We plan on using the algorithm for vector field registration of real world data and compare it to established algorithms with respect to reliability and runtime.


%\bibliography{sample}

\IfFileExists{\jobname.bbl}{}
 {\typeout{}
  \typeout{******************************************}
  \typeout{** Please run "bibtex \jobname" to optain}
  \typeout{** the bibliography and then re-run LaTeX}
  \typeout{** twice to fix the references!}
  \typeout{******************************************}
  \typeout{}
 }

 %\addcontentsline{toc}{section}{References}
  \bibliographystyle{unsrt} 
  \bibliography{./../Literaturverzeichnis}
\end{document}